\documentclass[11pt,oneside,reqno]{amsart}
\usepackage{bm,cancel,ulem}
\usepackage{amssymb,amsfonts}
\usepackage{amsmath,latexsym,enumerate}
\usepackage{mathrsfs}
\usepackage[unicode,breaklinks=true,colorlinks=true]{hyperref}
\usepackage[svgnames]{xcolor}

\usepackage[top=1in, bottom=1in, left=1.25in, right=1.25in, marginparwidth=1.1in, marginparsep=0.05in]{geometry}

\newtheorem{thm}{Theorem}[section]

\newtheorem{lem}[thm]{Lemma}

\theoremstyle{definition}
\newtheorem{defn}[thm]{Definition}
\newtheorem{rem}[thm]{Remark}
\numberwithin{equation}{section}
\newcommand{\eqdefa}{\overset{\mathrm{def}}{=\joinrel\!\!=}}

\newcommand{\R}{\mathbb R}
\newcommand{\Z}{\mathbb Z}

\newcommand{\cT}{\mathcal{T}}
\newcommand{\cF}{\mathcal{F}}

\newcommand{\dy}{\,{\rm d}y}
\newcommand{\dx}{\,{\rm d}x}

\newcommand{\dt}{\,{\rm d}t}
\newcommand{\ds}{\,{\rm d}s}

\newcommand{\dtau}{\,{\rm d}\tau}

\newcommand{\dxi}{\,{\rm d}\xi}
\newcommand{\andf}{\quad\hbox{and}\quad}

\newcommand{\du}{\,{\rm d}u}

\usepackage{accents}
\newcommand*{\dotc}[1] {\accentset{\mbox{\large\bfseries .}}{#1}}

\usepackage{marginnote}

\newcommand{\lec}{\lesssim}

\newcommand{\norm}[1]{\left\|#1\right\|}

\newcommand{\pd}{\partial}

\newcommand{\ep}{\varepsilon}

\renewcommand{\lec}{\lesssim}
\newcommand{\supp}{\operatorname{supp}}

\begin{document}
\title[Maximal regularity and caloric trace estimates]{Maximal regularity and caloric trace estimates in mixed Lebesgue norms for the heat equation}%
\author[H. Chen]{Hui Chen}%
\address[H. Chen]
{School of Science, Zhejiang University of Science and Technology, Hangzhou, 310023, China }
\email{chenhui@zust.edu.cn}
\author[S. Liang]{Su Liang}
\address[S. Liang]
{Department of Mathematics, University of British Columbia, Vancouver, BC V6T1Z2, Canada }
\email{liangsu96@gmail.com}
\author[T.-P. Tsai]{Tai-Peng Tsai}
\address[T.-P. Tsai]
{Department of Mathematics, University of British Columbia, Vancouver, BC V6T1Z2, Canada }
\email{ttsai@math.ubc.ca}

\date{}

\begin{abstract}
We study the heat equation in the half-space with nonhomogeneous Dirichlet boundary data. For the caloric extension $v$ of the boundary data $g$, we prove maximal regularity estimates in mixed Lebesgue norms $L^p_tL^q_x$  for any order derivative of $v$ in terms of mixed Besov and Lizorkin--Triebel type norms of $g$. We also establish the corresponding
reverse inequalities, which are caloric trace estimates recovering the boundary regularity of $g$ from the mixed-norm regularity of $v$.
As a model case, our results show that the natural
\[\dotc W^{1,p}\big(\R;L^q(\R^d_+)\big)\cap L^p\big(\R;\dotc W^{2,q}(\R^d_+)\big)\]
regularity norm of $v$ is controlled by the 
\[\dotc {F}^{1-\frac{1}{2q}}_{p,q}\big(\R;\,L^{q}(\R^{d-1})\big)\cap L^{p}\big(\R;\,\dotc{B}_{q,q}^{2-\frac{1}{q}}(\R^{d-1})\big)\]
norm of $g$. The maximal regularity estimate holds for $1\leq p,q<\infty$, while the caloric trace estimate holds for $1<p<\infty$ and $1\leq q\leq\infty$. In particular, the endpoint cases
$p=1$ or $q=1$ in the maximal regularity estimate are included and appear to be new. These endpoint estimates may be useful in the analysis of free-boundary Navier--Stokes problems with small initial data, whereas the caloric trace estimates may be relevant to the construction of Stokes or Navier--Stokes flows exhibiting strong boundary singularities.
\end{abstract}

\maketitle


\vskip 0.2cm


\tableofcontents

\section{Introduction}
This note is concerned with maximal regularity and trace estimates  for the heat equation. Let us first recall some known results.
For $1<p<\infty$ and sufficiently smooth domains $\Omega$ in $\R^d$, $d\ge2$, the classical trace theorem, first proved by Gagliardo \cite{Gagliardo-1957}, states that
\[
\norm{v|_{\pd\Omega}}_{W^{1-\frac{1}{p},p}(\pd\Omega)}
\lec
\norm{v}_{W^{1,p}(\Omega)}.
\]
Here $W^{1-\frac{1}{p},p}(\pd\Omega)$ denotes the Sobolev-Slobodeckij space, which agrees with the Besov space $B^{1-\frac{1}{p}}_{p,p}(\R^{d-1})$  in this range of $p$ if $\pd\Omega=\R^{d-1}$ \cite[\S 2.2.2, 2.3.5]{Triebel-1983}.
Gagliardo also proved the corresponding extension theorem for $1<p<\infty$, while in the endpoint case $p=1$ he showed that the natural trace space is
$L^1(\pd\Omega)$. We refer to \cite[\S 18.2, 18.4]{Leoni-2017} for an English account of these results.

We next recall the corresponding result for harmonic extensions in the half-space.  For $g\in C_c^\infty(\R^{d-1})$, its \textit{harmonic extension} is the harmonic function $v$ in $\R^d_+$ given by the convolution of $g$ with the Poisson kernel,
$v(x)=\int_{\R^{d-1}} \frac {2x_d\,g(y') \dy'} {|\pd B_1|\cdot |(y',0)-x|^d}$. For $1\leq p<\infty$,
\begin{equation*}
\norm{g}_{\dotc B^{1-\frac{1}{p}}_{p,p}(\R^{d-1})}
\lec
\norm{v}_{\dotc W^{1,p}(\R^d_+)}
\eqdefa
\sum_{i=1}^d \norm{\pd_{x_i}v}_{L^p(\R^d_+)}
\lec 
\norm{g}_{\dotc B^{1-\frac{1}{p}}_{p,p}(\R^{d-1})}.
\end{equation*}
See, e.g., Mironescu-Russ \cite[Theorems 1.1, 1.9, and 1.11]{Mironescu-Russ-2015}. Here the dot on top of $B$ and $W$ indicates homogeneous spaces. Hence the boundary homogeneous Besov norm of $g$ is equivalent to the homogeneous Sobolev norm of its harmonic extension.
The first inequality should be distinguished from the classical trace theorem. Indeed, due to the additional harmonicity constraint on $v$, the endpoint case $p=1$ is also included, whereas the corresponding classical trace estimate does not hold in general at $p=1$. In particular, there exist functions in $C_c^\infty(\R^{d-1})$ whose Besov seminorm $\norm{g}_{\dotc B^0_{1,1}}$ is infinite (see \cite[Proposition 5.8]{Mironescu-Russ-2015}), even though one can easily construct smooth extensions to $\R^d_+$ with finite $\dotc W^{1,1}(\R^d_+)$ norm. For this reason, we shall call the first inequality the \textit{harmonic trace estimate}.
We shall refer to the second inequality as the \textit{maximal regularity estimate} for the harmonic extension, since the harmonic trace estimate shows that it is optimal.

For time-dependent functions, it is natural to use mixed $L^p_tL^q_x$ type norms in the trace and extension estimates. This direction seems to have been initiated by Weidemaier \cite{Weidemaier-1994,Weidemaier-02,Weidemaier-05}.
 In particular, he proved the trace estimate \cite[Corollary 2.4]{Weidemaier-05}
\begin{align}\label{eq:Weidemaier-05}
\operatorname{tr}_{x_d=0}:\quad
&W^{1,p}\big(0,T;L^q(\R^d_+)\big)
\cap
L^p\big(0,T; W^{2,q}(\R^d_+)\big) \nonumber\\
&\to
F^{1-\frac{1}{2q}}_{p,q}
\big(0,T;\,L^q(\R^{d-1})\big)
\cap
L^p\big(0,T;\,B^{2-\frac{1}{q}}_{q,q}(\R^{d-1})\big),	
\end{align}
for $\frac32<q\leq p<\infty$. Here $F^{s}_{p,q}(\R;X)$ denotes the Bochner-Lizorkin-Triebel space; see Section \ref{sec:notation} for the definition.  Weidemaier also studied the corresponding extension estimate in \cite[Theorem 2.4]{Weidemaier-02}. 
The exponents have been relaxed to $1<p,q<\infty$ in Denk-Hieber-Pr\"uss \cite{DHP-07}.
Further developments in this direction include the work of Johnsen-Sickel \cite{Johnsen-Sickel-2008} on mixed-norm Lizorkin-Triebel spaces and that of Denk-Roodenburg \cite{Denk-Roodenburg-2025} on weighted Sobolev spaces.

Related mixed-norm estimates for parabolic initial-boundary value problems
were obtained by Denk-Hieber-Pr\"uss \cite{DHP-07}. In the model case of
the heat equation in the half-space,
\begin{equation}\label{eq:HBV}
	\left\{
	\begin{aligned}
		\,&\pd_t v-\Delta v=0, \quad x\in\R^d_+,\  t\in\R,\\[3pt]
		& v(x',x_d,t)|_{x_d=0}=g(x',t), \quad  x'\in\R^{d-1},\ t\in\R,
	\end{aligned}
	\right.
\end{equation}
a solution $v$ is given by the convolution of $g$ with the Poisson kernel for the heat equation,
\begin{equation}\label{eq:def-v}
	v(x,t)=-\int_\R\int_{\R^{d-1}}2\pd_{x_d} \Gamma(x'-y',x_d,t-s)\,g(y',s)\dy'\ds,
\end{equation}
where $\Gamma$ is the heat kernel
\[
\Gamma(x,t)= (4\pi t)^{-\frac{d}{2}}e^{\frac{-|x|^2}{4t}} \text { for } t>0;
\qquad
\Gamma(x,t)=0 \text { for } t\leq 0.
\]
When $d=1$,  \eqref{eq:def-v} reduces to
\begin{equation}\label{eq:def-v-1D}
	v(x,t)=-\int_\R2\pd_{x} \Gamma(x,t-s)\,g(s)\ds.
\end{equation}
We shall refer to this solution $v$ as the \textit{caloric extension} of $g$.
The results of \cite{DHP-07} imply that, for the caloric extension of $g$ with $\supp g\subset  [0,\infty)$ and $1<p,q<\infty$,
\begin{align}\label{eq:DHP}
	\norm{v}_{W^{1,p}\big(0,T;L^q(\R^d_+)\big)\cap L^p\big(0,T; W^{2,q}(\R^d_+)\big)}
	\lec
	\norm{g}_{{F}^{1-\frac{1}{2q}}_{p,q}\big(0,T;\,L^{q}(\R^{d-1})\big)\cap L^{p}\big(0,T;\,{B}_{q,q}^{2-\frac{1}{q}}(\R^{d-1})\big)}.
\end{align}
See \cite[Theorem 2.3]{DHP-07} and also
\cite[Proposition 1.1]{Ogawa-Shimizu-2022}. 
We call \eqref{eq:DHP} the caloric extension estimate, or equivalently the \textit{maximal regularity estimate} for the caloric extension. The corresponding reverse inequality, which estimates the boundary norm of $g$ by the interior mixed norm of its caloric extension, will be called the \textit{caloric trace estimate}.

The borderline case $p=1$, corresponding to $L^1$ regularity in time, is of particular interest. Such estimates do not follow from the standard maximal regularity theory in Banach spaces that satisfy the unconditional martingale differences condition
 (see \cite{DHP-07} and the references therein). Moreover, recent work shows that $L^1$-in-time regularity provides a useful framework for obtaining global-in-time control in fluid equations \cite{Danchin2025}. In this direction, Ogawa-Shimizu \cite{Ogawa-Shimizu-2022} studied $L^1$-in-time estimates for parabolic equations. These estimates were later used in their study of the global well-posedness of the Navier-Stokes
equations with free boundary conditions
\cite{Ogawa-Shimizu-2024-2,Ogawa-Shimizu-2024}. Their results imply that, for the caloric extension $v$ of $g$,
\begin{equation}\label{eq:Ogawa-Shimizu}
	\norm{v}_{\dotc W^{1,1}\big(\R;L^q(\R^d_+)\big)\cap L^1\big(\R; \dotc W^{2,q}(\R^d_+)\big)}
	\lec
	\norm{g}_{{\dotc F}^{1-\frac{1}{2q}}_{1,1}\big(\R;\,L^{q}(\R^{d-1})\big)\cap L^{1}\big(\R;\,{B}_{q,1}^{2-\frac{1}{q}}(\R^{d-1})\big)},
\end{equation}
for $1\leq q<\infty$; see \cite[Theorem 2.2]{Ogawa-Shimizu-2022}.
Compared with the boundary spaces appearing in \eqref{eq:DHP}, the second lower indices in both the Lizorkin-Triebel and Besov norms are $1$ rather than $q$. Furthermore, the Besov norm appearing in the last term of \eqref{eq:Ogawa-Shimizu} is inhomogeneous, whereas the regularity norm on the left-hand side is homogeneous. This suggests that the boundary regularity in the above estimate may not be optimal. These observations motivate the estimates proved in the present paper. 

For caloric trace estimates, one example is \cite[Theorem 7.1]{Ogawa-Shimizu-2022}: for $1\le q\le\infty$,
\begin{equation}\label{eq:Ogawa-Shimizu-trace}
\norm{g}_{{\dotc F}^{1-\frac{1}{2q}}_{1,1}\big(\R;\,\dotc B^0_{q,1}(\R^{d-1})\big)\cap L^{1}\big(\R;\,\dotc B^{2-1/q}_{q,1}(\R^{d-1})\big)}\lec
\norm{\pd_tv}_{L^{1}\big(\R; \dotc B^0_{q,1}(\R^d_+)\big)}+\norm{\nabla^2 v}_{L^{1}\big(\R; \dotc B^0_{q,1}(\R^d_+)\big)}.
\end{equation}
It shows the optimality of the maximal regularity estimate in \cite[Theorem 2.1]{Ogawa-Shimizu-2022}.
There is no corresponding inequality with $\dotc B^0_{q,1}$ on the right-hand side replaced by $L^q$.

This paper is a continuation of our earlier work \cite{Chen-Liang-Tsai-2025-2}, where we considered the equal exponent case $p=q$ for both maximal regularity estimates and caloric trace estimates. That work, in turn, extends the one-dimensional first-order caloric trace estimates and maximal regularity estimates established in \cite[Lemma 2.8]{Chen-Liang-Tsai-2025}, 
which were motivated by the blow-up examples of Chang-Kang \cite{Chang-Kang-2023}.
The case $p=q$ is more direct: possibly reflecting the identification $F^s_{p,p}=B^s_{p,p}$, the natural trace space is the \textit{anisotropic Besov space}, rather
than the intersection space appearing on the right-hand side of \eqref{eq:DHP}. In the setting of anisotropic Besov spaces,
see \cite[Corollary 1.2]{Chen-Liang-Tsai-2025-2} for the corresponding maximal regularity and caloric trace estimates, and Amann \cite[Theorem 4.5.2]{Amann-2009} for related trace estimates. As we will point out in a comment after Theorem \ref{thm:dd-max-reg}, the  intersection space appearing on the right-hand side of \eqref{eq:regu-est2} (homogeneous version of \eqref{eq:DHP}) for general $p,q$ reduces to the anisotropic Besov space when $p=q$. We also note that caloric trace estimates are related to the construction of singular solutions to the Stokes and Navier-Stokes equations near the boundary; see, for instance, 
\cite[Theorems 1.1 and 1.6]{Chang-Kang-2023} and
\cite[Theorems 1.2 and 1.4]{Chen-Liang-Tsai-2025}.

To overcome the more difficult case when $p,q$ are not necessarily equal, the new ingredients used in this paper are maximal functions and duality arguments. In particular, we will use the Fefferman-Stein vector-valued maximal inequality and a vector-valued version of Peetre's maximal inequality for frequency-localized functions; see Lemmas \ref{lem:FS} and \ref{lem:peetre}. They are well suited to the Littlewood-Paley frequency decompositions used in the Lizorkin-Triebel and Besov norms.

In this paper, we study mixed-norm maximal regularity estimates, modeled on \eqref{eq:DHP}, and the corresponding caloric trace estimates for pairs $(v,g)$ satisfying \eqref{eq:HBV}. Our results show that the homogeneous version of \eqref{eq:DHP} holds for $1\leq p,q<\infty$, thereby improving \eqref{eq:DHP} and \eqref{eq:Ogawa-Shimizu}. We also prove the corresponding caloric trace estimates for $1<p<\infty$ and $1\leq q\leq\infty$, which quantify the sharpness of the maximal regularity estimates.

\begin{thm}[One-dimensional case]\label{thm:1D}
Let $d=1$ and $n=\alpha+2\beta$ with  nonnegative integers
$\alpha,\beta$. Suppose that  $g\in \mathcal{S}(\R)$ 
and let $v(x,t)$ be defined by \eqref{eq:def-v-1D}.
For $1\leq p,q<\infty$, we have 
\begin{equation}\label{eq:max-reg-1D}
\norm{\pd_x^{\alpha}\pd_t^{\beta}v}_{L^{p}\big(\R;\,L^{q}(\R_+)\big)}
\lesssim_{p,q,n}
\norm{g}_{\dotc {F}^{\frac{n}{2}-\frac{1}{2q}}_{p,q}(\R)}.
\end{equation}
Moreover, for $1<p<\infty$ and $1\leq q\leq \infty$, as well as  $(p,q)=(1,1)$ and $(p,q)=(\infty,\infty)$, we have
\begin{equation}\label{eq:trace-1D}
\norm{g}_{\dotc {F}^{\frac{n}{2}-\frac{1}{2q}}_{p,q}(\R)}\lec_{p,q,n} 	\norm{\pd_x^{n}v}_{L^{p}\big(\R;\,L^{q}(\R_+)\big)}.
\end{equation} 
\end{thm}

\textit{Comments on Theorem \ref{thm:1D}}:
\begin{enumerate}[(i)]
\item The exponent $\frac{n}{2}-\frac{1}{2q}$ can be non-positive. It happens when $n=0$ or when $n=q=1$. 
\item When $p=q$, Theorem \ref{thm:1D} recovers the case $d=1$ of
\cite[Theorem 1.1]{Chen-Liang-Tsai-2025-2}.
\item For $1\leq p,q<\infty$, $s\in \R$, the space $\mathcal{S}_0(\R)$ of functions in $\mathcal{S}(\R)$ whose Fourier transforms are supported away from $0$ is dense in $\dotc {F}^{s}_{p,q}(\R)$. This follows by the same argument as \cite[Proposition 2.27]{Bahouri-Chemin-Danchin-2011}.
\end{enumerate}

\begin{thm}[Maximal regularity estimates in higher dimensions]\label{thm:dd-max-reg}
Let $d\geq 2$, $1\leq p,q<\infty$ and $n=|\alpha|+2 \beta=\sum_{i=1}^d\alpha_i+2\beta$ with nonnegative integers $\alpha=(\alpha_1,\ldots,\alpha_d)$ and $\beta$. 
Suppose that  $g\in \mathcal{S}(\R^{d-1}\times \R)$ and let $v(x,t)$ be defined by \eqref{eq:def-v}.
Then we have that for $n=0$,
\begin{equation}\label{eq:regu-est1}
\norm{v}_{L^{p}\big(\R;\,L^{q}(\R_+^{d})\big)}\lec_{p,q,d,n}
\norm{g}_{\dotc {F}^{-\frac{1}{2q}}_{p,q}\big(\R;\,L^{q}(\R^{d-1})\big)},
\end{equation}
and for $n\geq1$,
\begin{equation}\label{eq:regu-est2}
\norm{\pd_x^{\alpha}\pd_t^{\beta}v}_{L^{p}\big(\R;\,L^{q}(\R_+^{d})\big)}\lec_{p,q,d,n}
\norm{g}_{\dotc {F}^{\frac{n}{2}-\frac{1}{2q}}_{p,q}\big(\R;\,L^{q}(\R^{d-1})\big)} +\norm{g}_{L^{p}\big(\R;\,\dotc{B}_{q,q}^{n-\frac{1}{q}}(\R^{d-1})\big)}.
\end{equation}
\end{thm}

\textit{Comments on Theorem \ref{thm:dd-max-reg}}:
\begin{enumerate}[(i)]
\item The special case $n=2$ of \eqref{eq:regu-est2} gives
\[
\norm{|\nabla_x^2 v| + |\pd_t v|}_{L^{p}\big(\R;\,L^{q}(\R_+^{d})\big)}
\lec 
\norm{g}_{\dotc {F}^{1-\frac{1}{2q}}_{p,q}\big(\R;\,L^{q}(\R^{d-1})\big)} +\norm{g}_{L^{p}\big(\R;\,\dotc{B}_{q,q}^{2-\frac{1}{q}}(\R^{d-1})\big)}.
\]
For $1<p<\infty$, it recovers \eqref{eq:DHP}, proved in Denk-Hieber-Pr\"uss
\cite{DHP-07}, with weaker homogeneous norms on the right-hand side, and adds the borderline exponent $q=1$.

\item For $n=2$, $p=1$, \eqref{eq:regu-est2} recovers \eqref{eq:Ogawa-Shimizu}, proved in Ogawa-Shimizu \cite{Ogawa-Shimizu-2022}, and the second lower indices of both terms on the right-hand side are weakened from 1 to $q$.

\item In the special case $p=q$ and $n>\frac{1}{q}$, the right-hand side of \eqref{eq:regu-est2} is equivalent to the anisotropic Besov norm
$\norm{g}_{\dotc B^{n-\frac 1q, \frac n 2 - \frac 1{2q}}_{q,q} (\R^{d-1}\times \R)}$ by Chang-Kang \cite[Proposition 2.1]{Chang-Kang-2023}. Hence the combination of Theorems 
\ref{thm:dd-max-reg} and \ref{thm:dd-trace} recovers the case $d\ge2$ of
\cite[Theorem 1.1]{Chen-Liang-Tsai-2025-2} when $n>\frac{1}{q}$.

\item Estimates \eqref{eq:regu-est1} with $n=0$ and \eqref{eq:regu-est2} with $n=1$ seem to be new. 

\item  For \(1\leq p,q<\infty\), as in the one-dimensional case, any function with suitable decay (see Definitions \ref{def:S'_h}-\ref{def:spaces}) and finite norm given by the right-hand side of \eqref{eq:regu-est1} or \eqref{eq:regu-est2} can be approximated by functions in the Schwartz class. See Lemma \ref{lem:Schwartz-density} and Lemma \ref{lem:Schwartz-density-2}.
\end{enumerate}

\begin{thm}[Caloric trace estimates in higher dimensions]\label{thm:dd-trace}
Let $d\geq 2$. 
Suppose that  $g\in \mathcal{S}(\R^{d-1}\times \R)$ and let $v(x,t)$ be defined by \eqref{eq:def-v}.
For $1<p<\infty$ and $1\leq q\leq \infty$, as well as $p=q=1$ and $p=q=\infty$, and the integer $n>\frac{1}{q}$, we have
\begin{equation}\label{thm:dd-trace-eq1}
\norm{g}_{\dotc {F}^{\frac{n}{2}-\frac{1}{2q}}_{p,q}\big(\R;\,L^{q}(\R^{d-1})\big)}+\norm{g}_{L^{p}\big(\R;\,\dotc{B}_{q,q}^{n-\frac{1}{q}}(\R^{d-1})\big)}
\lec_{p,q,d,n} \norm{\pd_{x_d}^nv}_{L^{p}\big(\R;\,L^{q}(\R_+^{d})\big)}.
\end{equation}
\end{thm}

\textit{Comments on Theorem \ref{thm:dd-trace}}:
\begin{enumerate}[(i)]
\item This trace estimate shows the optimality of \eqref{eq:regu-est2} and the equivalence of the two sides when $1< p<\infty$, $1\le q <\infty$ (or $p=q=1$) and $n>\frac{1}{q}$.

\item Unlike Theorems \ref{thm:1D} and \ref{thm:dd-max-reg}, this theorem assumes
the condition $n>\frac{1}{q}$. This condition is also assumed in the time-independent case 
\cite[Theorem 1.1]{Mironescu-Russ-2015} and is valid in \eqref{eq:Ogawa-Shimizu-trace}.

\item We have actually proved more general estimates than \eqref{thm:dd-trace-eq1}. We can bound
\begin{align*}
	\norm{\left(2^{(\frac{n}{2}-\frac{1}{2q})k}\|\Delta_{\leq \frac{k}{2}}^{h}\Delta_{k}^{\cT} g\|_{L^{q}(\R^{d-1})}\right)_{k\in\Z}}_{L^p(\R;\ell^q(\Z))}\\
+\norm{\left(2^{(n-\frac{1}{q})m}\|\Delta_{m}^h\Delta_{<2m}^{\cT} g\|_{L^{q}(\R^{d-1})}\right)_{m\in\Z}}_{L^p(\R;\ell^q(\Z))}
\end{align*}
by $\norm{\pd_{x_d}^nv}_{L^{p}\big(\R;\,L^{q}(\R_+^{d})\big)}$ without the condition $n>\frac{1}{q}$,
and bound
\begin{align*}
	\norm{\left(2^{(\frac{n}{2}-\frac{1}{2q})k}\|\Delta_{>\frac{k}{2}}^{h}\Delta_{k}^{\cT} g\|_{L^{q}(\R^{d-1})}\right)_{k\in\Z}}_{L^p(\R;\ell^q(\Z))}\\
+\norm{\left(2^{(n-\frac{1}{q})m}\|\Delta_{m}^h\Delta_{\geq 2m}^{\cT} g\|_{L^{q}(\R^{d-1})}\right)_{m\in\Z}}_{L^p(\R;\ell^q(\Z))}
\end{align*}
assuming $n>\frac{1}{q}$.
See Lemmas \ref{lem:trace-d0} and \ref{lem:trace-d0b}.

\item For $1\leq p\leq \infty$, $1\leq q <\infty$ and $n>\frac{1}{q}$, it follows from taking $L^p$-norm in time of
the classical trace estimate of Uspenski\u{\i} \cite{Uspenskii}, for general functions in $\R^d_+$
given in Mironescu-Russ \cite[Theorem 1.1]{Mironescu-Russ-2015}, that
\[
\norm{g}_{L^{p}\big(\R;\,\dotc{B}_{q,q}^{n-\frac{1}{q}}(\R^{d-1})\big)}
\lec  \sum_{|\alpha|=n} \norm{\pd_{x}^\alpha v}_{L^{p}\big(\R;\,L^{q}(\R_+^{d})\big)}.
\]
It includes borderline cases $p=1,\infty$, $1<q<\infty$, and $(p,q)=(\infty,1)$ not covered in 
\eqref{thm:dd-trace-eq1}.
\end{enumerate}

The paper is organized as follows. In Section \ref{sec:notation}, we introduce the notation used throughout the paper and collect some preliminary lemmas. In Section \ref{sec:1D}, we prove Theorem \ref{thm:1D}. In Section \ref{sec:DD}, we prove Theorems \ref{thm:dd-max-reg} and \ref{thm:dd-trace}.

\section{Notation and preliminaries}\label{sec:notation}
For two quantities $X$ and $Y$, we write $X \lec_{\sigma} Y$ if $X \leq C(\sigma)\, Y$
for some constant $C(\sigma)>0$ depending on parameters $\sigma=(\sigma_1,\sigma_2,\dots)$. We often suppress this dependence and simply write $X \lec Y$ when it is clear from the context.
Unless otherwise indicated, the implicit constants may depend on $p$, $q$, $d$, $n$, but not on the functions being estimated.

Denote by $\cF$ the Fourier transform, with subscripts indicating the variables when necessary. For $x\in \R^d_+$, we write $x=(x',x_d)$ with $x'\in \R^{d-1}$ and $x_d\in \R_+$, and
\[
\hat f\eqdefa \cF_{x',t} f(\xi',x_d,\tau)
= \frac{1}{(2\pi)^{\frac{d}{2}}}
\int_{\R^{d-1}\times\R}
e^{-i(x'\cdot\xi' + t\tau)} f(x,t)\,dx'\,dt,
\]
that is, $\hat f$ denotes the Fourier transform of $f$ in the $(x',t)$ variables.

Denote
\begin{equation}\label{def-delta}
\Delta_{m}^{h}\eqdefa\mathcal{F}_{x'}^{-1}\varphi(2^{-m}\xi')\mathcal{F}_{x'} \andf \Delta_{k}^{\cT}\eqdefa\mathcal{F}_{t}^{-1}\theta(2^{-k}\tau)\mathcal{F}_{t},\quad m,k\in \mathbb{Z}
\end{equation}
as the homogeneous Littlewood-Paley operators in the horizontal and temporal directions, respectively. Here $\theta\in C_c^{\infty}\Big(\big(\frac{3}{4},\frac{8}{3}\big)\cup\big(-\frac{8}{3},-\frac{3}{4}\big)\Big)$ is an even function satisfying
\[\forall \tau\in \R\setminus \{0\}, \quad \sum_{j\in \mathbb{Z}}\theta(2^{-j}\tau)=1,\]
and $\varphi(\xi')=\theta(|\xi'|)$.
Denote (note the slight asymmetry in the definitions)
\begin{equation*}
	\varphi_0(\xi')\eqdefa\sum_{j\leq0}\varphi(2^{-j}\xi'),
	\quad	
	\theta_0(\tau)\eqdefa\sum_{j<0}\theta(2^{-j}\tau),
\end{equation*}
with the convention that $\varphi_0(0)=\theta_0(0)=1$. Denote
\begin{equation*}
	\Delta_{\leq m}^{h}\eqdefa\sum_{j\leq m}\Delta_{j}^{h}=\mathcal{F}_{x'}^{-1}\varphi_0(2^{-m}\xi')\mathcal{F}_{x'},
	\quad 
	\Delta_{<k}^{\cT}\eqdefa\sum_{j<k}\Delta_{j}^{\cT}=\mathcal{F}_{t}^{-1}\theta_0(2^{-k}\tau)\mathcal{F}_{t}.
\end{equation*}
We further set
$\Delta_{\leq \frac{k}{2}}^{h}\eqdefa \Delta_{\leq \lfloor\frac{k}{2}\rfloor }^{h}$,
where $\lfloor \cdot \rfloor$ denotes the floor function, and define
\[\Delta_{> \frac{k}{2}}^{h}\eqdefa 1-\Delta_{\leq \frac{k}{2}}^{h}, 
\quad 
\Delta_{\geq k}^{\cT}\eqdefa1-\Delta_{<k}^{\cT}.\]
Denote
\[
\tilde{\varphi}(\xi') \eqdefa \sum_{j=-1}^{1} \varphi(2^{-j}\xi'), 
\quad
\tilde{\theta}(\tau) \eqdefa \sum_{j=-1}^{1} \theta(2^{-j}\tau),
\quad
\tilde{\theta}_0(\tau) \eqdefa \theta_0(\tau) + \theta(\tau),
\]
which satisfy
\begin{equation}\label{def-tilde}
	\left\{	\begin{aligned}
		\ & \tilde{\varphi}=1 &&\text{on}\quad\supp \varphi, \\
		\ & \tilde{\theta}=1 &&\text{on}\quad\supp \theta,\\
		\ & \tilde{\theta}_0=1&&\text{on}\quad\supp\theta_0.
	\end{aligned}\right.
\end{equation}

Denote by $\mathcal S$ the Schwartz space and by $\mathcal S'$ the space of tempered distributions.
Motivated by Bahouri-Chemin-Danchin
\cite[Definition~1.26]{Bahouri-Chemin-Danchin-2011},
we define the following spaces by imposing low-frequency conditions
separately in the horizontal and temporal variables. See also Cobb \cite{Cobb-2024} for a discussion of different conventions for the space $\mathcal S'_h$ and of the distinction between these realization spaces and the quotient space $\mathcal S'/\mathcal P$.

\begin{defn}\label{def:S'_h}
For $E\in\{\R^{d-1}, \R^{d-1}\times\R\}$, we define
\begin{align*}
	\mathcal{S}'_{h,x'}(E)\eqdefa
	\{g\in \mathcal{S}'(E):\,\, \Delta_{\leq m}^h g\to 0 \hbox{ in } \mathcal{S}'(E) \hbox{ as }m\to -\infty\}.
\end{align*}
For $E\in\{\R, \R^{d-1}\times\R\}$, we define
\begin{align*}
	\mathcal{S}'_{h,t}(E)\eqdefa
	\{g\in \mathcal{S}'(E):\,\, \Delta_{<k}^{\mathcal{T}} g\to 0 \hbox{ in } \mathcal{S}'(E) \hbox{ as }k\to -\infty\}.
\end{align*}
\end{defn}

In the spirit of
\cite[Definition~2.15]{Bahouri-Chemin-Danchin-2011}, we use
$\mathcal S'_{h,x'}$ and $\mathcal S'_{h,t}$ as the ambient distribution spaces for the homogeneous function spaces below.
The defining low-frequency conditions imply the reconstruction formulas
$
g=\sum_{m\in\mathbb Z}\Delta_m^h g
$
in $\mathcal S'$ for $g\in \mathcal{S}'_{h,x'}$,
and 
$
g=\sum_{k\in\mathbb Z}\Delta_k^{\mathcal T}g
$
in $\mathcal S'$ for $g\in \mathcal{S}'_{h,t}$.
Consequently, the homogeneous quantities introduced below are genuine
norms, rather than merely seminorms, on the corresponding ambient spaces.
Here and throughout, $\norm{(a_m)_{m\in\Z}}_{\ell^q(\Z)}$ denotes the $\ell^q$ norm of the sequence $(a_m)_{m\in\Z}$, and $\norm{\cdot}_{L^{q}(\R^d)}$ denotes the Lebesgue norm.

\begin{defn}\label{def:spaces}
Let $s\in \R$. For $1\leq p,q,r\leq \infty$, the \textit{homogeneous Besov space} 
$\dotc{B}_{p,q}^{s}(\R^{d-1})$
consists of all
$g\in S'_{h,x'}(\R^{d-1})$ such that
\begin{equation*}
\norm{g}_{\dotc{B}_{p,q}^{s}(\R^{d-1})}
\eqdefa 
\norm{\left(2^{ms}\norm{\Delta_{m}^{h}g}_{L^{p}(\R^{d-1})}\right)_{m\in\Z}}_{\ell^{q}(\Z)}<\infty,
\end{equation*}
and the \textit{Bochner-Besov space} $L^{r}\big(\R;\,\dotc{B}_{p,q}^{s}(\R^{d-1})\big)$ consists of all $g\in S'_{h,x'}(\R^{d-1}\times \R)$ such that 
\begin{equation*}
	\norm{g}_{L^{r}\big(\R;\,\dotc{B}_{p,q}^{s}(\R^{d-1})\big)}
	\eqdefa 
	\norm{\,\norm{g(\cdot,t)}_{\dotc{B}_{p,q}^{s}(\R^{d-1})}\,}_{L^{r}(\R)}<\infty.
\end{equation*}
For $1\leq p<\infty$ and $1\leq q,r \leq \infty$,  the \textit{homogeneous Lizorkin-Triebel space} $\dotc{F}_{p,q}^{s}(\R)$ consists of all $g\in S'_{h,t}(\R)$ such that 
\begin{equation*}
	\norm{g}_{\dotc{F}_{p,q}^{s}(\R)}\eqdefa \norm{\left(2^{ks}\Delta_{k}^{\mathcal{T}}g\right)_{k\in\Z}}_{L^{p}(\R;\ell^{q}(\Z))}<\infty,
\end{equation*}
and the \textit{Bochner-Lizorkin-Triebel space} 
$\dotc{F}^{s}_{p,q}\big(\R;L^r(\R^{d-1})\big)$ 
consists of all $g\in S'_{h,t}(\R^{d-1}\times \R)$ such that 
\begin{equation*}
\norm{g}_{\dotc{F}^{s}_{p,q}\big(\R;L^r(\R^{d-1})\big)}
\eqdefa 
\norm{\left(2^{ks}\norm{\Delta_{k}^{\cT}g}_{L^r(\R^{d-1})}\right)_{k\in\Z}}_{L^{p}(\R;\ell^{q}(\Z))}<\infty.
\end{equation*} 
\end{defn}

Note that $\dotc{F}_{p,q}^{s}(\R)$ needs a different definition for $p=\infty>q$, see \cite[2.3.1 Remark 4, 2.3.4]{Triebel-1983}. We will use $\dotc{F}_{\infty,\infty}^{s}=\dotc{B}_{\infty,\infty}^{s}$.
Note that, in proving the trace estimate \eqref{eq:Weidemaier-05}, Weidemaier used the finite-difference definition of Bochner-Lizorkin-Triebel spaces in \cite{Weidemaier-05}. This definition is equivalent to the above Littlewood-Paley type definition in a certain range of the indices $p$ and $q$; see \cite[2.5.10]{Triebel-1983}.

We next establish some density results for these homogeneous spaces. Since the corresponding one-dimensional results are classical, we focus on the case $d\geq2$.

\begin{lem}\label{lem:Schwartz-density}
Let $d\geq 2$, $s\in\mathbb R$ and $1\leq p,q,r<\infty$.

\begin{enumerate}[(i)]

\item Suppose that
$g\in \dotc F^s_{p,q}\big(\mathbb R;L^r(\mathbb R^{d-1})\big)$.
Then there exists a sequence $\{g_m\}_{m\geq1}\subset \mathcal S(\mathbb R\times\mathbb R^{d-1})$ such that
\[
\lim_{m\to\infty}
\norm{g-g_m}_{\dotc F^s_{p,q}\bigl(\mathbb R;L^r(\mathbb R^{d-1})\bigr)}
=0.
\]

\item Suppose that
$g\in L^p\bigl(\mathbb R;\dotc B^s_{r,q}(\mathbb R^{d-1})\bigr)$.
Then there exists a sequence $\{g_m\}_{m\geq1}\subset \mathcal S(\mathbb R\times\mathbb R^{d-1})$ such that
\[
\lim_{m\to\infty}
\norm{g-g_m}_{L^p\bigl(\mathbb R;\dotc B^s_{r,q}(\mathbb R^{d-1})\bigr)}
=0.
\]

\end{enumerate}
\end{lem}

\begin{proof}
We first prove (i). For $N\geq 1$, let
\[
g_N=\sum_{|k|\leq N}\Delta_k^{\mathcal T}g.
\]
Since $q<\infty$, the tails of the Littlewood-Paley decomposition
converge to zero. Therefore,
\[
\norm{g-g_N}_{\dotc F^s_{p,q}
\bigl(\mathbb R;L^r(\mathbb R^{d-1})\bigr)}
\to 0,
\quad\text{as }\,N\to\infty.
\]
Moreover, for each fixed $N$, the function $g_N$ has only finitely
many Littlewood-Paley frequencies in the time variable, and
\[
g_N\in L^p\bigl(\mathbb R;L^r(\mathbb R^{d-1})\bigr).
\]
Since $p,r<\infty$, for each fixed $N$, we may choose
$\phi_{N,m}\in\mathcal S(\mathbb R\times\mathbb R^{d-1})$ such that
\[
\norm{\phi_{N,m}-g_N}_{L^p
\bigl(\mathbb R;L^r(\mathbb R^{d-1})\bigr)}
\to 0,
\quad\text{as }\, m\to\infty.
\]
Let $\chi_N\in C_c^\infty(\mathbb R\setminus\{0\})$ be equal to $1$
on the time-frequency support of $g_N$, and set
\[
\widetilde\phi_{N,m}
=\mathcal F_t^{-1}
\bigl(\chi_N(\tau)\mathcal F_t(\phi_{N,m})\bigr).
\]
Then
$\widetilde\phi_{N,m}\in
\mathcal S(\mathbb R\times\mathbb R^{d-1})$, and
$\widetilde\phi_{N,m}-g_N$ has time frequencies contained in a fixed
compact subset of $\mathbb R\setminus\{0\}$. Hence,
\[
\norm{\widetilde\phi_{N,m}-g_N}_{\dotc F^s_{p,q}
\bigl(\mathbb R;L^r(\mathbb R^{d-1})\bigr)}
\lesssim_N
\norm{\phi_{N,m}-g_N}_{L^p
\bigl(\mathbb R;L^r(\mathbb R^{d-1})\bigr)}
\to 0,
\quad\text{as }\,m\to\infty.
\]
By first choosing $N$ sufficiently large and then choosing $m$
sufficiently large, we obtain a diagonal sequence that converges to
$g$ in
$\dotc F^s_{p,q}
\bigl(\mathbb R;L^r(\mathbb R^{d-1})\bigr)$.
This proves (i).

For (ii), we first truncate the Littlewood-Paley decomposition in
the spatial variables by setting
\[
g_N=\sum_{|j|\leq N}\Delta_j^h g.
\]
We then use the density of Schwartz functions in
$L^p\bigl(\mathbb R;L^r(\mathbb R^{d-1})\bigr)$.
Applying a smooth spatial Fourier cutoff that is equal to $1$ on the
spatial-frequency support of $g_N$, the same finite-frequency
argument gives convergence in
$
L^p\bigl(\mathbb R;\dotc B^s_{r,q}(\mathbb R^{d-1})\bigr)$.
This proves (ii).
\end{proof}

\begin{lem}\label{lem:Schwartz-density-2}
Let $d\geq2$, $1\leq p,q<\infty$, and $s\geq0$. If $s=0$,
assume in addition that $q=1$.
Suppose $g$ satisfies
$$
\norm{g}_{\dotc F^s_{p,q}
\bigl(\R;L^q(\R^{d-1})\bigr)}
+
\norm{g}_{L^p
\bigl(\R;\dotc B^{2s}_{q,q}(\R^{d-1})\bigr)}
<\infty.
$$
Then there exists a sequence
$
\{g_m\}_{m\geq 1}
\subset\mathcal S(\R^{d-1}\times\R)
$
such that
$$
\norm{g-g_m}_{\dotc F^s_{p,q}
\bigl(\R;L^q(\R^{d-1})\bigr)}
+
\norm{g-g_m}_{L^p
\bigl(\R;\dotc B^{2s}_{q,q}(\R^{d-1})\bigr)}
\to 0,
\quad\text{as }\,m\to\infty.
$$
\end{lem}

\begin{proof}
For simplicity, denote
$$
A(f)
\eqdefa
\norm{f}_{\dotc F^s_{p,q}
\bigl(\R;L^q(\R^{d-1})\bigr)}
\quad\text{and}\quad
B(f)
\eqdefa
\norm{f}_{L^p
\bigl(\R;\dotc B^{2s}_{q,q}(\R^{d-1})\bigr)}.
$$

$\bullet$ Case 1: $s>0$. 

\noindent For $N\geq1$, let
$$
g_N
=
\bigl(\Delta_{<N}^{\mathcal T}\Delta_{\leq N}^h
-\Delta_{<-N}^{\mathcal T}\Delta_{\leq-N}^h\bigr)g.
$$
We claim that
\begin{equation}\label{claim01}
A(g-g_N)+B(g-g_N)\to 0,\quad\text{as }\,N\to\infty.
\end{equation}

For the high-frequency part,
\begin{align*}
A\bigl((1-\Delta_{<N}^{\mathcal T}\Delta_{\leq N}^h)g\bigr)
&\leq
A\bigl((1-\Delta_{<N}^{\mathcal T})g\bigr)
+
A\bigl(\Delta_{<N}^{\mathcal T}
(1-\Delta_{\leq N}^h)g\bigr)\notag\\[3pt]
&\lec A\bigl((1-\Delta_{<N}^{\mathcal T})g\bigr)
+
A\bigl(
(1-\Delta_{\leq N}^h)g\bigr)\to 0,\quad\text{as }\,N\to\infty.
\end{align*}
The first term on the right-hand side converges to zero by the summability of the Littlewood-Paley pieces, while the second term converges to zero by an application of the dominated convergence theorem.
Similarly, for any positive integer $N_0$,
\begin{align*}
&\quad\ B\bigl((1-\Delta_{<N}^{\mathcal T}\Delta_{\leq N}^h)g\bigr)\\[3pt]
&\leq
B\bigl((1-\Delta_{\leq N}^h)g\bigr)
+
B\bigl(\Delta_{\leq N}^h
(1-\Delta_{<N}^{\mathcal T})g\bigr)\\[3pt]
&\lec B\bigl((1-\Delta_{\leq N}^h)g\bigr)
+
B\bigl(
(1-\Delta_{<N}^{\mathcal T})g\bigr)\\[3pt]
&\lec  B\bigl((1-\Delta_{\leq N}^h)g\bigr)
+B\bigl(\sum_{|j|>N_0}\Delta_{j}^h
g\bigr)
+
B\bigl(\sum_{|j|\leq N_0}\Delta_{j}^h
(I-\Delta_{<N}^{\mathcal T})g\bigr).
\end{align*}
For each fixed $N_0$, the first and third terms on the right-hand side converge to zero as $N\to\infty$. On the other hand, the second term can be made arbitrarily small by choosing $N_0$ sufficiently large. Therefore, by first letting $N\to\infty$ and then $N_0\to\infty$, we obtain
\begin{equation*}
B\bigl((1-\Delta_{<N}^{\mathcal T}\Delta_{\leq N}^h)g\bigr)\to0,
\quad\text{as }\,N\to\infty.
\end{equation*}
We next estimate the low-frequency part. By Young's inequality,
$$
A\bigl(
\Delta_{<-N}^{\mathcal T}\Delta_{\leq-N}^h g
\bigr)
\lec\norm{\left(2^{ks}\norm{\Delta_k^{\mathcal T}g}_{L^q(\R^{d-1})}\right)_{k\leq-N+2}}_{L^p(\R;\ell^q(\mathbb{Z}))}.
$$
Since $q<\infty$, the right-hand side tends to zero as $N\to\infty$.
Similarly, by Minkowski's inequality and Young's inequality,
$$
B\bigl(\Delta_{<-N}^{\mathcal T}\Delta_{\leq-N}^h g\bigr)\lec\norm{\left(2^{2sj}\norm{\Delta_j^h g}_{L^q(\R^{d-1})}\right)_{j\leq-N+2}}_{L^p(\R;\ell^q(\mathbb{Z}))},
$$
which also tends to zero. This proves the claim \eqref{claim01}.

For each fixed $N$, the Fourier transform of $g_N$ is supported in a compact set of the form 
$$
\left\{
|\tau|\lesssim 2^N,\quad |\xi'|\lec 2^N,\quad
|\tau|\gtrsim 2^{-N}
\ \text{or}\
|\xi'|\gtrsim 2^{-N}
\right\}.
$$
Using a smooth partition of unity on this set, we may write
$
g_N=g_N^{\mathcal T}+g_N^h$,
where the time Fourier support of $g_N^{\mathcal T}$ is contained in
$
\{2^{-N}\lesssim|\tau|\lesssim 2^N\}
$,
and the spatial Fourier support of $g_N^h$ is contained in
$
\{2^{-N}\lesssim|\xi'|\lesssim 2^N\}
$.
The multipliers used in this decomposition are bounded on both spaces. Therefore,
$$
\norm{g_N^{\mathcal T}}_{L^p(\R;L^q(\R^{d-1}))}\lec_N A(g_N)
\quad\text{and}\quad
\norm{g_N^h}_{L^p(\R;L^q(\R^{d-1}))}
\lec_N B(g_N).
$$
Consequently,
$
g_N\in L^p\bigl(\R;L^q(\R^{d-1})\bigr)
$.
Thus, for each fixed $N$, there exist
$\phi_{N,m}\in\mathcal S(\R^{d-1}\times\R)$ such that
$
\norm{\phi_{N,m}-g_N}_{L^p(\R;L^q(\R^{d-1}))}
\to 0
$ as $m\to\infty$.
Choose a smooth compactly supported Fourier multiplier $\Psi_N(\xi',\tau)$ that is equal to $1$ on the Fourier support of $g_N$, and define
$$
\widetilde\phi_{N,m}
=
\mathcal F^{-1}_{x',t}
\bigl(\Psi_N(\xi',\tau)
\mathcal F_{x',t}(\phi_{N,m})\bigr).
$$
Using $\ell^1\hookrightarrow\ell^q$ and Young's inequality, we obtain
\begin{align*}
A(\widetilde\phi_{N,m}-g_N)
&\lesssim
\sum_{k\leq N+10}2^{ks}
\norm{\Delta_k^{\mathcal T}
(\widetilde{\phi}_{N,m}-g_N)}_{L^p(\R;L^q(\R^{d-1}))}\\[3pt]
&\lec_N
\norm{\widetilde{\phi}_{N,m}-g_N}_{L^p(\R;L^q(\R^{d-1}))}\\[3pt]
&\lec_N
\norm{\phi_{N,m}-g_N}_{L^p(\R;L^q(\R^{d-1}))}\to0,\quad\text{as }\, m\to\infty,
\end{align*}
where, in the second line, we have used the summability of the low-frequency weights due to $s>0$.  The estimate for $B\bigl(\widetilde{\phi}_{N,m}-g_N\bigr)$ is similar. Choosing a suitable diagonal sequence completes the proof when $s>0$.

$\bullet$ Case 2: $s=0$ and $q=1$.

\noindent Define
$$
g_N=
(\Delta_{<N}^{\mathcal T}-\Delta_{<-N}^{\mathcal T})
(\Delta_{\leq N}^h-\Delta_{\leq-N}^h)g.
$$
We claim again that \eqref{claim01} holds. The high-high-frequency and low-low-frequency terms can be treated exactly as in Case 1. It remains only to estimate the additional high-low-frequency terms arising from the new definition of $g_N$. More precisely, it suffices to show that
\begin{equation}\label{low0}
A(\Delta_{\leq-N}^h g)+B(\Delta_{<-N}^{\mathcal T}g)\to0,\quad\text{as }\, N\to\infty.
\end{equation}
We first consider the first term in \eqref{low0}. Since
$
\norm{g}_{L^p
\bigl(\R;\dotc B^0_{1,1}(\R^{d-1})\bigr)}
<\infty
$,
we have that for every fixed $k$, 
$$
\Delta_k^{\mathcal T}g
\in
L^\infty\bigl(\R;\dotc B^0_{1,1}(\R^{d-1})\bigr).
$$
Therefore,
$\Delta_k^{\mathcal T}g(t,\cdot)$ belongs to
$\dotc B^0_{1,1}(\R^{d-1})$ for almost every $t$. For such $t$,
$$
\Delta_{\leq-N}^h\Delta_k^{\mathcal T}g(t,\cdot)
\to 0
\quad\text{in }L^1(\R^{d-1}),
\quad\text{as }\,N\to\infty.
$$
Since
$$
\norm{\Delta_{\leq-N}^h\Delta_k^{\mathcal T}g(t,\cdot)}_{L^1(\R^{d-1})}
\lec\norm{\Delta_k^{\mathcal T}g(t,\cdot)}_{L^1(\R^{d-1})},
$$
the dominated convergence theorem gives
\begin{equation}\label{low1}
A(\Delta_{\leq-N}^h g)\to 0,
\quad\text{as }\,N\to\infty.
\end{equation}
We next consider the second term in \eqref{low0}. If $p>1$, the low-frequency operator
$\Delta_{<-N}^{\mathcal T}$ converges strongly to zero on
$
L^p\bigl(\R;\dotc B^0_{1,1}(\R^{d-1})\bigr) 
$. Indeed, let $\theta(t)=\norm{g(\cdot,t)}_{\dotc B^0_{1,1}(\R^{d-1})}\in L^{p}(\R)$ and choose $\theta_m(t)\in C_c^{\infty}(\R)$ such that $\theta_m\to\theta$ in $L^p(\R)$. By Young's inequality, we get
\begin{equation*}
B(\Delta_{<-N}^{\mathcal T}g)\lec 2^{-N(1-1/p)} \norm{\theta_m}_{L^{1}(\R)}+\norm{\theta-\theta_m}_{L^{p}(\R)}.
\end{equation*}
Therefore, by first letting $N\to\infty$ and then $m\to\infty$, we obtain that for $p>1$,
\begin{equation}\label{low2}
	B(\Delta_{<-N}^{\mathcal T} g)\to 0, \quad\text{as }\,N\to\infty.
\end{equation}
For $p=1$, the proof is similar to \eqref{low1}. Actually,
\begin{equation*}
	\norm{
		\Delta_{<-N}^{\mathcal T}\Delta_j^h g
	}_{L^1(\R\times\R^{d-1})}\lec \sum_{k<-N}	\norm{
	\Delta_{k}^{\mathcal T} g
	}_{L^1(\R\times\R^{d-1})}\to0,\quad\text{as }\,N\to \infty.
\end{equation*}
Since
$$
\norm{
\Delta_{<-N}^{\mathcal T}\Delta_j^h g
}_{L^1(\R\times\R^{d-1})}
\lesssim
\norm{\Delta_j^h g}_{L^1(\R\times\R^{d-1})},
$$
another application of the dominated convergence theorem yields \eqref{low2}. 
Combining \eqref{low1} and \eqref{low2}, we get \eqref{low0}, which leads to \eqref{claim01}.

Finally, for each fixed $N$, the Fourier transform of $g_N$ is supported in compact annuli away from zero in both the time and spatial variables. Hence,
$
g_N\in L^p\bigl(\R;L^1(\R^{d-1})\bigr)
$.
Moreover, for every function $f$ whose Fourier support is contained in the same fixed frequency region,
$$
A(f)+B(f)
\lesssim_N
\norm{f}_{L^p(\R;L^1(\R^{d-1}))}.
$$
Approximating $g_N$ in $L^p(\R;L^1(\R^{d-1}))$ by Schwartz functions and then applying a smooth Fourier cutoff supported in slightly larger time and spatial annuli, we obtain Schwartz functions converging to $g_N$ simultaneously in both norms. A diagonal argument completes the proof.

\end{proof}

If $f\in L^1_{\mathrm{loc}}(\R^d)$, the \textit{Hardy-Littlewood maximal function} is defined by
\begin{equation*}
Mf(x) \eqdefa \sup_{r>0} \frac{1}{|B_r(x)|}\int_{B_r(x)} |f(y)|\,dy,
\end{equation*}
where $B_r(x)\subset \R^d$ is the ball centered at $x$ with radius $r$.
It is well known \cite[p.~13]{Stein-93} that for $1<p\leq \infty$,
\begin{equation}\label{Maximal1}
\norm{Mf}_{L^p(\R^d)} \lec \norm{f}_{L^p(\R^d)},
\end{equation}
while for $p=1$ only the weak-type estimate holds. Moreover, we have
\[
(Mf)^r\leq M(|f|^r) \,\,\hbox{ if } r\geq 1,\quad   M(|f|^r)\leq (Mf)^r \,\,\hbox{ if } 0<r< 1.
\]   
These inequalities will be used frequently without further mention.

We recall the following classical vector-valued maximal inequality, originally due to Fefferman-Stein \cite[Theorem~1]{Fefferman-Stein-1971}.

\begin{lem}\label{lem:FS}
Let $1<p<\infty$ and $1<q\leq \infty$, or $p=q=\infty$. There exists a constant $C=C(p,q,d)$ such that for any sequence $\{f_k\}_{k\in\mathbb{Z}} \subset L^1_{\mathrm{loc}}(\R^d)$,
\begin{equation}\label{eq:FS}
\norm{(M f_k)_{k\in\Z}}_{L^{p}(\R^d;\ell^{q}(\Z))}
\leq C\, \norm{(f_k)_{k\in\Z}}_{L^{p}(\R^d;\ell^{q}(\Z))}.
\end{equation}
\end{lem}

\begin{rem}\label{rem2.3}
The case $1<p\leq \infty$ and $q=\infty$ in \eqref{eq:FS} is not covered
in \cite{Fefferman-Stein-1971}, but it follows directly from the pointwise
inequality $\norm{(Mf_k)_{k\in\Z}}_{\ell^\infty(\Z)}\leq M\norm{(f_k)_{k\in\Z}}_{\ell^\infty(\Z)}$. However, it fails for $p=\infty$ and $1\le q<\infty$. For example,
for fixed $1\le q<\infty$, let $f_k=\chi_{I_k}$, where $I_k=(2^{k-1},2^k)$ and $k\in\Z$. Then $\left(\sum_{k\in\Z}|f_k(x)|^q\right)^{\frac{1}{q}}\leq 1$ for all $x\in\R$. On the other hand,
	\begin{align*}
		Mf_k(0)\ge \frac{1}{2\cdot2^k}\int_{0}^{2^k}f_k(y)\dy= \frac{1}{4},
	\end{align*}
	and hence $\left(\sum_{k\in\Z}|Mf_k(0)|^q\right)^{\frac{1}{q}}=\infty$.
\end{rem}

The following lemma provides a scale-uniform pointwise bound for convolutions in terms of the Hardy-Littlewood maximal function. 

\begin{lem}\label{lem:HLmax-uni}
Let $h:\R^d\to \R$ be a function satisfying
$|h(x)|\lec (1+|x|)^{-N}$  
for some $N>d$, and $h_j(x)\eqdefa2^{jd}h(2^j x)$ for $j\in \Z$. Then we have
\begin{equation}\label{eq:maximal-uniform}
   \sup_{j\in\Z} |h_j *f(x)| \lec  Mf(x).
\end{equation}
\end{lem}
\begin{proof}
The proof follows from Stein \cite[(16) in p.~57]{Stein-93} as $\Phi(x)= \sup_{|y|>|x|} |h(y)|$ is an integrable, radially non-increasing function.
\end{proof}
\begin{rem}
From Lemma \ref{lem:HLmax-uni}, it follows that
\[\sup_{j\in\Z}|\Delta_j f(x)| \lec Mf(x),\quad \sup_{j\in\Z}|\Delta_{\leq j}f(x)| \lec  M f(x).
\]
In contrast, the estimate
	\[ |h *f(x)| \lec\, (M|f|^r(x))^{\frac{1}{r}}\] 
cannot hold when $0<r<1$. A counterexample is $f=\chi_{B(y_0,\ep)}$, where $y_0\neq0$ is chosen so that $h(-y_0)\neq 0$, and $0<\ep\ll|y_0|$. Then, at $x=0$, we have $(M|f|^r(0))^{\frac{1}{r}}\lec \ep^{d/r}$ while $\ep^d\lec |h *f(0)|$.
\end{rem}

The following lemma states that Peetre's maximal function (the left-hand side of \eqref{eq:peetre}) is bounded pointwise by a power of the Hardy-Littlewood maximal function. Its idea originates from Fefferman-Stein \cite{Fefferman-Stein-1972} and Peetre \cite{Peetre-1975}. A detailed proof under the assumption $f\in \mathcal{S}$ can be found in Triebel \cite[Theorem 1.3.1]{Triebel-1983}, while a version valid for more general $f\in \mathcal{S}'$ is given in Li \cite[Lemma 2.3]{LiDong-2019}. 
In Lemma \ref{lem:peetre}, we record a variant of \cite{Triebel-1983,LiDong-2019}, originally stated for functions of $t\in\mathbb R^n$, in a form suitable for functions of $(x,t)\in\mathbb R^d\times\mathbb R$, with the absolute value $|\cdot|$ replaced by a Banach norm $\norm{\cdot}_X$ for functions on $\mathbb R^d$.

\begin{lem}\label{lem:peetre}
Assume that the Fourier transform of a function $f(x,t)$ in the $t$-variable satisfies
\[
\supp \mathcal F_t f \subset \R^d\times B_R,
\qquad
B_R\eqdefa\{\tau\in\R:|\tau|<R\}.
\]
Let $\norm{\cdot}_{X}$ be a Banach norm acting on the $x$-variable,
and set $F(t)=\norm{f(\cdot,t)}_{X}$. 
Then, we have that for fixed $r\in(0,\infty)$,
	\begin{equation}\label{eq:peetre}
		\sup_{s\in \R}\frac{F(t-s)}{1+|Rs|^{\frac{1}{r}}}\lec \left(M|F|^{r}(t)\right)^{\frac{1}{r}},
	\end{equation}
provided that the left-hand side of \eqref{eq:peetre} is finite in the case $0<r<1$.
\end{lem}

The boundedness of the left-hand side of \eqref{eq:peetre} is an essential condition and should be verified when the lemma is applied with $0<r<1$. When we bound $\sup_{s\in \R}\frac{F(t-s)}{1+|Rs|^{2}}\lec  MF(t)$ in later sections, we can drop the exponent $2$ and use $r=1$ case (since $\frac{1}{1+|Rs|^2}\lec \frac{1}{1+|Rs|}$), avoiding checking the condition. 
However, when we use the vector-valued version \eqref{eq:tri83-2} of Lemma \ref{lem:tri83-2}, to cover the endpoint cases $p=1$ or $q=1$, it is
sometimes necessary to take $r<1$.

\begin{proof}
We follow the argument in \cite{LiDong-2019}. By translation it suffices to prove the case $t=0$ and by scaling one can assume $R=1$. Since $\supp \mathcal F_t f \subset \R^d\times B_R$, we may choose $h\in \mathcal{S}(\mathbb{R})$ such that $\mathcal{F}_th$ is equal to the appropriate nonzero constant on $B_1$, so that
\begin{align*}
	f(x,s)=\int_\R h(s-u)f(x,u)\du,\quad |h(s)|\leq \frac{C(h,N)}{(1+|s|)^N},
\end{align*}
for any $N>1$.

Consider first the case $r\geq 1$.  For fixed $s$, we decompose $\R=\bigcup_{m=0}^\infty A_m$, where $A_0=\{u:\,|s-u|\leq (1+|s|)\}$ and $A_{m}=\{u:\,2^{m-1}(1+|s|)< |s-u|\leq 2^{m}(1+|s|)\}$ for $m\geq 1$. Also let $B_{m}=\{u:\,|u|\leq 2^{m+1}(1+|s|)\}$ for $m\geq 0$. Note $A_m$ and $B_m$ have different centers, and $A_m\subset B_m$. Then
\begin{align}\label{eq:lem:dyadic}
 F(s)&\leq \sum_{m=0}^{\infty}\int_{A_m} |h(s-u)|F(u)\du \notag\\[3pt]
	&\leq \int_{B_0}|h(s-u)|F(u)\du+\sum_{m=1}^{\infty}\frac{C_N}{2^{N(m-1)}(1+|s|)^{N}}\int_{B_m} F(u)\du\notag\\[3pt]
	&\lec \left(\int_{B_0}F(u)^r\du\right)^{\frac{1}{r}}+ \sum_{m=1}^{\infty}\frac{2^{m+2}(1+|s|)}{2^{N(m-1)}(1+|s|)^{N}} M|F|(0)\notag\\[3pt]
	&\lec (1+|s|)^{\frac{1}{r}} \left(M|F|^{r}(0)\right)^{\frac{1}{r}}+M|F|(0),
\end{align}
where we choose $N$ large enough so that the series converges. Since $M|F|(0)\leq \left(M|F|^{r}(0)\right)^{\frac{1}{r}}$ for $r\geq 1$, we obtain 
\[\frac{F(s)}{(1+|s|)^{\frac{1}{r}}}\lec \left(M|F|^{r}(0)\right)^{\frac{1}{r}},
\]
which proves \eqref{eq:peetre}.

Next we consider the case $0<r<1$.  
\begin{align*}
 F(s)&\leq \int_\R |h(s-u)|F(u)^r\cdot F(u)^{1-r}\du\notag\\[3pt]
	&\leq \int_\R |h(s-u)|F(u)^r\cdot (1+|u|)^{\frac{1-r}{r}}\du\cdot \left(\sup_{u\in \R}\frac{F(u)}{(1+|u|)^{\frac{1}{r}}}\right)^{1-r}\notag\\[3pt]
	&\leq(1+|s|)^{\frac{1-r}{r}}\int_\R |h(s-u)|F(u)^r\cdot (1+|s-u|)^{\frac{1-r}{r}}\du\cdot \left(\sup_{u\in \R}\frac{F(u)}{(1+|u|)^{\frac{1}{r}}}\right)^{1-r}.
\end{align*}
By the same dyadic annular decomposition as in \eqref{eq:lem:dyadic}, we have 
\begin{align*}
	&\,\quad \int_\R |h(s-u)|F(u)^r\cdot (1+|s-u|)^{\frac{1-r}{r}}\du\\[3pt]
	&\lec\int_{B_0}F(u)^r\du+\sum_{m=1}^{\infty}\frac{1}{2^{(N-\frac{1-r}{r})(m-1)}(1+|s|)^{(N-\frac{1-r}{r})}}\int_{B_m} F(u)^r\du\\[3pt]
	&\lec (1+|s|)M|F|^{r}(0)+M|F|^r(0).
\end{align*}
Hence we obtain 
\[F(s)\lec (1+|s|)^{\frac{1}{r}}M|F|^{r}(0)\cdot \left(\sup_{u\in \R}\frac{F(u)}{(1+|u|)^{\frac{1}{r}}}\right)^{1-r}.\]
The desired inequality then follows provided that  $\sup_{u\in \R}\frac{F(u)}{(1+|u|)^{\frac{1}{r}}}$ is finite.
\end{proof}

The following lemma combines Lemmas \ref{lem:FS} and \ref{lem:peetre}, and extends \cite[Theorem 1.6.2]{Triebel-1983}.

\begin{lem}\label{lem:tri83-2}
	Let $0<p<\infty$ and $0<q\leq\infty$, or $p=q=\infty$, and  $0<r<\min (p,q)$.
	Suppose that $f_j$, for $j\in\mathbb{Z}$,  satisfy the assumptions of Lemma \ref{lem:peetre} with radius $R_j$, and $F_j(t)=\norm{f_j(\cdot,t)}_{X}$ for the same $X$-norm. Then, we have 
	\begin{equation}\label{eq:tri83-2}
		\norm{\left(\sup_{s\in \R}\frac{F_j(t-s)}{1+|R_js|^{\frac{1}{r}}}\right)_{j\in\Z}}_{L^{p}(\R;\ell^q(\Z))}	\lec \norm{(F_j(t))_{j\in\Z}}_{L^{p}(\R;\ell^q(\Z))},
	\end{equation}
provided that $\sup_{s\in \R}\frac{F_j(t-s)}{1+|R_js|^{\frac{1}{r}}}$ is finite for each $j$ in the case $0<r<1$.
\end{lem}
\begin{proof}
By \eqref{eq:peetre}, we have 
\begin{align*}
\norm{\left(\sup_{s\in \R}\frac{F_j(t-s)}{1+|R_js|^{\frac{1}{r}}}\right)_{j\in\Z}}_{L^{p}(\R;\ell^q(\Z))}
&\lec \norm{\left(\left(M|F_j|^{r}(t)\right)^{\frac{1}{r}}\right)_{j\in\Z}}_{L^{p}(\R;\ell^q(\Z))}\\[3pt]
&=\norm{\left(M|F_j|^{r}(t)\right)_{j\in\Z}}_{L^{\frac{p}{r}}(\R;\ell^{\frac{q}{r}}(\Z))}^{\frac{1}{r}},
\end{align*}
which together with \eqref{eq:FS} yields
\[ 
\norm{\left(\sup_{s\in \R}\frac{F_j(t-s)}{1+|R_js|^{\frac{1}{r}}}\right)_{j\in\Z}}_{L^{p}(\R;\ell^q(\Z))}	\lec \norm{\left(|F_j|^{r}(t)\right)_{j\in\Z}}_{L^{\frac{p}{r}}(\R;\ell^{\frac{q}{r}}(\Z))}^{\frac{1}{r}}\\[3pt]
=\norm{(F_j(t))_{j\in\Z}}_{L^{p}(\R;\ell^q(\Z))}.
\]
Hence we get \eqref{eq:tri83-2}.
\end{proof}

The following calculus lemma will be useful in later sections.
\begin{lem}
Let $1\leq q<\infty$, $R>1$, $a,b\in\R$, and $c>0$. We have that for $a+\frac{1}{q^2}>0$ and any sequence $A_j$, 
\begin{equation}\label{eq:int-xd1}
\int_0^\infty\Big(\sum_{j\in\mathbb{Z}} x^{a}R^{bj}e^{-cxR^{j}}|A_j|\Big)^q \dx\lec \sum_{j\in\mathbb{Z}}R^{(b-a)jq-j}|A_j|^q.
\end{equation}
\end{lem}

\begin{proof}
It is straightforward to verify that for $r>0$,
\begin{equation}\label{eq:lem-xd4}
	\sup_{x>0} \sum_{j\in\mathbb{Z}}\left[x R^{j}e^{-cx R^{j}} \right]^{r}\lec 1,
\end{equation}
and for $\sigma>-1$,
\begin{equation}\label{eq:lem-xd5}
	\int_0^\infty x^{\sigma} e^{-cxR^{j}}\dx\lec R^{-j(\sigma+1)}.
\end{equation}	

For $q=1$, \eqref{eq:int-xd1} is trivial, so in the following we assume $1<q<\infty$. 
By using H\"{o}lder's inequality, we have
\begin{align*}
\Big(	\sum_{j\in\mathbb{Z}}R^{bj}e^{-cxR^{j}}|A_j|\Big)^q&=\Big(\sum_{j\in\mathbb{Z}}\Big[R^{(b-1)j}\cdot \left(R^{j}e^{-cx R^{j}} \right)^{1-\frac{q-1}{q^2}} |A_j|\Big]\cdot\left[R^{j}e^{-c xR^{j}} \right]^{\frac{q-1}{q^2}}\Big)^q\\[3pt]
&\lec \Big(\sum_{j\in\mathbb{Z}}\Big[R^{(b-1)j}\cdot \left(R^{j}e^{-cx R^{j}} \right)^{1-\frac{q-1}{q^2}} |A_j|\Big]^q\Big)\cdot\Big(\sum_{j\in\mathbb{Z}}\left[R^{j}e^{-cx R^{j}} \right]^{\frac{1}{q}}\Big)^{q-1},
\end{align*}
which together with \eqref{eq:lem-xd4} and \eqref{eq:lem-xd5} yields (using $aq^2+1>0$)
	\begin{align*}
		\int_0^\infty\Big(	\sum_{j\in\mathbb{Z}}x^{a}R^{bj}e^{-cxR^{j}}|A_j|\Big)^q\dx&\lesssim \sum_{j\in\mathbb{Z}}\int_{0}^{+\infty} R^{(bq-1+\frac{1}{q})j} x^{aq-1+\frac{1}{q}}e^{-cxR^j(q-1+\frac{1}{q})} \dx\cdot |A_j|^q \\[3pt]
		&\lesssim \sum_{j\in\mathbb{Z}}\ R^{(b-a)jq-j} |A_j|^q.
	\end{align*}
Hence \eqref{eq:int-xd1}.
\end{proof}

\section{The one-dimensional case}\label{sec:1D}

In this section, we prove Theorem \ref{thm:1D}, the one-dimensional case. We first establish the maximal regularity estimate \eqref{eq:max-reg-1D} in Lemma \ref{lem:max-reg-1D}, and then prove the trace estimate \eqref{eq:trace-1D} in Lemmas \ref{lem:trace-1d-1}-\ref{lem:trace-1d-2}.

\subsection{Maximal regularity estimate}

\begin{lem}\label{lem:max-reg-1D}
Under the assumptions of Theorem \ref{thm:1D}, estimate \eqref{eq:max-reg-1D} holds.
\end{lem}

\begin{proof}
For $v$ defined in \eqref{eq:def-v-1D}, we have (see \cite[Corollary 2.5]{Chen-Liang-Tsai-2025})
\begin{equation}\label{eq:1D-v}
\cF_t v=  e^{-x\sqrt{i\tau}} \cdot \cF_tg.
\end{equation}
Here we take the principal branch of the square root, so that the real part $\hbox{Re} \sqrt{i \tau} = \sqrt{|\tau|/2}$ for $\tau \in \R$. Hence
\begin{align}\label{eq:1D-0}
\pd_x^{\alpha}\pd_t^{\beta}v=\sum_{k\in \Z}\Delta_{k}^{\mathcal{T}}\pd_x^{\alpha}\pd_t^{\beta}v&=\sum_{k\in \Z} \mathcal{F}_{t}^{-1}\left(P(x,\tau) \cdot\theta(2^{-k}\tau)\cdot \cF_t g\right)\notag\\[3pt]
&=\sum_{k\in \Z} \mathcal{F}_{t}^{-1}\left(P(x,\tau) \cdot\tilde{\theta}(2^{-k}\tau)\cdot\theta(2^{-k}\tau)\cdot\cF_t g\right),
\end{align}
where $\Delta_{k}^{\mathcal{T}}$ and $\tilde{\theta}$ are defined in \eqref{def-delta} and \eqref{def-tilde} respectively, and
\begin{equation*}
P(x,\tau)\eqdefa(-\sqrt{i\tau})^{\alpha}\cdot(i \tau)^\beta e^{-x\sqrt{i\tau}}.
\end{equation*}

Now we claim that there exists a constant $c>0$ such that
\begin{equation}\label{eq:1d-1}
|\mathcal{F}_{t}^{-1}\left(P(x,\tau) \cdot \tilde{\theta}(2^{-k}\tau)\right)(t)|\lec \frac{2^k\cdot 2^{\frac{nk}{2}}}{\left(1+|2^k t|^2\right)^2}e^{-cx 2^{\frac{k}{2}}}.
\end{equation}
The reader may refer to \cite[Lemma 2.4]{Bahouri-Chemin-Danchin-2011} for the underlying idea of this estimate. We include a detailed derivation here, and will use similar estimates repeatedly later without further comment.
Notice that
\begin{align*}
\mathcal{F}_{t}^{-1}\left(P(x,\tau) \cdot \tilde{\theta}(2^{-k}\tau)\right)(t)&=\frac{1}{(2\pi)^{1/2}}\int_{\R} e^{i t\cdot\tau}\cdot P(x,\tau)\cdot \tilde{\theta}(2^{-k}\tau) \dtau\\[3pt]
&=\frac{2^k}{(2\pi)^{1/2}}\int_{\R} e^{i2^k t\cdot\tau}\cdot P(x,2^k\tau)\cdot \tilde{\theta}(\tau)  \dtau.
\end{align*}
By repeated integration by parts using $e^{i2^k t\cdot\tau} = (1+|2^k t|^2)^{-1}(1-\pd_{\tau}^2)e^{i2^k t\cdot\tau}$, we obtain
\begin{equation}\label{claim1}
|\mathcal{F}_{t}^{-1}\left(P(x,\tau) \cdot \tilde{\theta}(2^{-k}\tau)\right)(t)|\lec \frac{2^k}{\left(1+|2^k t|^2\right)^2}\cdot  \norm{\left(1-\pd_{\tau}^2\right)^2\left(P(x,2^k\tau)\cdot \tilde{\theta}(\tau)\right)}_{L^{1}(\R)}.
\end{equation}
Since the real part $\hbox{Re}(\sqrt{i2^k\tau})=\frac{2^{\frac{k}{2}}}{\sqrt 2}\sqrt {|\tau|}$ regardless of whether $\tau$ is positive or negative, we deduce
\begin{align*}
|\left(1-\pd_{\tau}^2\right)^2\left(P(x,2^k\tau)\cdot \tilde{\theta}(\tau)\right)|\lec 2^{\frac{nk}{2}}e^{-cx 2^{\frac{k}{2}}}.
\end{align*}
Given the compact support of $\tilde{\theta}$, the corresponding $L^1$ norm satisfies the same bound, which together with \eqref{claim1} leads to \eqref{eq:1d-1}.
	
Therefore, in view of \eqref{eq:1D-0} and \eqref{eq:1d-1}, we have
\begin{equation*}
|\Delta_{k}^{\mathcal{T}}\pd_x^{\alpha}\pd_t^{\beta}v| \lec \int_{\R}\frac{2^k\cdot 2^{\frac{nk}{2}}}{\left(1+|2^k u|^2\right)^2}e^{-cx 2^{\frac{k}{2}}}\cdot|\Delta_k^{\mathcal{T}}g(t-u)|\du\lec2^{\frac{nk}{2}}e^{-cx 2^{\frac{k}{2}}}\cdot \sup_{u\in \R} \frac{|\Delta_k^{\mathcal{T}}g(t-u)|}{1+|2^ku|^2},	
\end{equation*}
which together with \eqref{eq:int-xd1} implies that
\begin{align*}
		\int_0^\infty |\pd_x^{\alpha}\pd_t^{\beta}v|^q \dx &\lec  \int_0^\infty \Big(\sum_{k\in \Z}2^{\frac{nk}{2}}e^{-cx 2^{\frac{k}{2}}}\cdot \sup_{u\in \R} \frac{|\Delta_k^{\mathcal{T}}g(t-u)|}{1+|2^ku|^2}\Big)^q  \dx\\[3pt]
		&\lec \sum_{k\in \Z} 2^{\frac{nkq-k}{2}}\cdot \Big(\sup_{u\in \R} \frac{|\Delta_k^{\mathcal{T}}g(t-u)|}{1+|2^ku|^2}\Big)^q.
	\end{align*}
Since $g\in \dotc {F}^{\frac{n}{2}-\frac{1}{2q}}_{p,q}(\R)$, $\Delta_k^{\mathcal{T}}g$ is bounded for each $k$ (by the Bernstein inequality $\norm{\Delta_k^{\mathcal{T}}g}_{L^\infty}\lec_k \norm{\Delta_k^{\mathcal{T}}g}_{L^p}$), thus the quantity $\sup_{u\in \R} \frac{|\Delta_k^{\mathcal{T}}g(t-u)|}{1+|2^ku|^2}$ is finite. By applying \eqref{eq:tri83-2} with $r=\frac12$, we arrive at \eqref{eq:max-reg-1D}.\end{proof}

\subsection{Caloric trace estimate}
In this subsection, we present the proof of the caloric trace estimate \eqref{eq:trace-1D} in the following three lemmas (Lemmas \ref{lem:trace-1d-1}-\ref{lem:trace-1d-2}), each corresponding to a different range of the indices $p,q$.

\begin{lem}\label{lem:trace-1d-1}
Under the assumptions of Theorem \ref{thm:1D} with $1<p,q<\infty$, estimate \eqref{eq:trace-1D} holds.
\end{lem}
\begin{proof}
In view of \eqref{eq:1D-v}, we have
\begin{equation}\label{eq:3-5}
\cF_tg(\tau)=4\int_{0}^\infty x(-\sqrt{i\tau})^{2-n}\cdot e^{-x\sqrt{i\tau}}\cdot \pd_x^n\cF_tv(x,\tau)\dx.
\end{equation}
By an argument similar to that used for \eqref{eq:1d-1}, we get
\begin{align}\label{eq:lem-trace-2}
		\left|\Delta_{k}^{\mathcal{T}}g\right|
		&\lec\left|\int_{0}^\infty \cF_t^{-1}\left(x(-\sqrt{i\tau})^{2-n}\cdot e^{-x\sqrt{i\tau}}\cdot\theta(2^{-k}\tau)\cdot \pd_x^n\cF_tv(x,\tau)\right)\dx\right|\notag\\[3pt]
		&\lec\int_{\R}\int_{0}^\infty\frac{x  2^{\frac{4-n}{2}k}\cdot e^{-cx 2^{\frac{k}{2}}}}{\left(1+|2^k (t-s)|^2\right)^2}\cdot\left|\pd_x^nv(x,s)\right|\dx\ds.
\end{align}

We next use a duality argument. 
For any sequence $\{\psi_k(t)\}_{k\in\mathbb Z}$ satisfying
$$
\norm{(\psi_k(t))_{k\in\Z}}_{L^{p'}(\mathbb R;\ell^{q'}(\Z))}<\infty,
\qquad
\frac1p+\frac1{p'}=\frac1q+\frac1{q'}=1,
$$
we have
\begin{equation}\label{Dual1}
	\int_{\R} \Delta_{k}^{\mathcal{T}}g\cdot\psi_{k}(t)\dt=\int_{\R} \Delta_{k}^{\mathcal{T}}g\cdot\tilde{\Delta}_{k}^{\mathcal{T}}\psi_{k}(t)\dt,
\end{equation}
where
\begin{equation*}
	\tilde{\Delta}_{k}^{\mathcal{T}}\psi_{k}\eqdefa\cF_t^{-1}\left(\tilde{\theta}(2^{-k}\tau)\cdot\cF_t\psi_{k}\right).
\end{equation*}
By virtue of  \eqref{eq:lem-trace-2} and \eqref{Dual1}, we infer
\begin{align}
K_1&\eqdefa \sum_{k\in\mathbb{Z}}2^{(\frac{n}{2}-\frac{1}{2q})k}\int_{\R} \Delta_{k}^{\mathcal{T}}g\cdot\psi_{k}(t)\dt\notag\\[3pt]
&\lec\sum_{k\in\mathbb{Z}}\int_{\R}\int_{\R}\int_{0}^{\infty}
\frac{x2^{\frac{4-\frac{1}{q}}{2}k}\cdot e^{-cx 2^{\frac{k}{2}}}}{\left(1+|2^k (t-s)|^2\right)^2}\cdot\left|\pd_x^nv(x,s)\right|\cdot\left|\tilde{\Delta}_{k}^{\mathcal{T}}\psi_{k}(t)\right|\dx\ds\dt\notag\\[3pt]
&\lec\int_{\R}\int_{0}^{\infty}\sum_{k\in\mathbb{Z}}\Big( x2^{\frac{2-\frac{1}{q}}{2}k}\cdot e^{-cx 2^{\frac{k}{2}}}\cdot\sup_{s\in \R} \frac{|\tilde{\Delta}_{k}^{\mathcal{T}}\psi_{k}(t-s)|}{1+|2^ks|^2}\Big)\cdot\left|\pd_x^nv(x,t)\right|\dx\dt.\label{M1}
\end{align} 
By H\"older's inequality and \eqref{eq:int-xd1} with $1<q<\infty$, we deduce
\begin{align*}
	K_1&\lec\int_{\R}\Big(\sum_{k\in\mathbb{Z}}\Big(\sup_{s\in \R} \frac{|\tilde{\Delta}_{k}^{\mathcal{T}}\psi_{k}(t-s)|}{1+|2^ks|^2}\Big)^{q'}\Big)^{\frac{1}{q'}}\cdot\norm{\pd_x^nv(\cdot,t)}_{L^q\left(\R_+\right)}\dt\\[3pt]
	&\lec\norm{\left(\sup_{s\in \R} \frac{|\tilde{\Delta}_{k}^{\mathcal{T}}\psi_{k}(t-s)|}{1+|2^ks|^2}\right)_{k\in\Z}}_{L^{p'}(\R;\ell^{q'}(\Z))}\cdot\norm{\pd_x^nv}_{L^{p}(\R;L^q\left(\R_+\right))}.
\end{align*}
Thus, applying \eqref{eq:tri83-2} with $r=1$, \eqref{eq:maximal-uniform} and \eqref{eq:FS} in succession,
we obtain
\begin{equation*}
K_1
\lec \norm{(\psi_{k})_{k\in\Z}}_{L^{p'}(\R;\ell^{q'}(\Z))}\cdot\norm{\pd_x^nv}_{L^{p}(\R;L^q\left(\R_+\right))}.
\end{equation*}
By a duality argument, it follows that for $1<p<\infty$ and $1< q<\infty$,
\begin{equation*}
\norm{\left(2^{(\frac{n}{2}-\frac{1}{2q})k} \Delta_{k}^{\mathcal{T}}g\right)_{k\in\Z}}_{L^{p}(\R;\ell^{q}(\Z))}\lec\norm{\pd_x^nv}_{L^{p}(\R;L^q\left(\R_+\right))},
\end{equation*}
which leads to \eqref{eq:trace-1D}.
\end{proof}

We next deal with two endpoint cases.

\begin{lem}\label{lem32}
Under the assumptions of Theorem \ref{thm:1D} with $1\leq p<\infty$ and $q=1$, estimate \eqref{eq:trace-1D} holds.
\end{lem}
\begin{proof}
The estimate \eqref{M1} is still valid for $q=1$. Hence
\begin{align*}
	K_1&\leq \int_{\R}\int_{0}^{\infty}\sum_{k\in\mathbb{Z}}\Big( x2^{\frac{k}{2}}\cdot e^{-cx 2^{\frac{k}{2}}}\Big)\cdot \sup_{k}\Big(\sup_{s\in \R} \frac{|\tilde{\Delta}_{k}^{\mathcal{T}}\psi_{k}(t-s)|}{1+|2^ks|^2}\Big)\cdot\left|\pd_x^nv(x,t)\right|\dx\dt \\[3pt]
	&\lec \int_{\R} \sup_{k}\Big(\sup_{s\in \R} \frac{|\tilde{\Delta}_{k}^{\mathcal{T}}\psi_{k}(t-s)|}{1+|2^ks|^2}\Big)\cdot\norm{\pd_x^nv(\cdot,t)}_{L^1(\R_+)}\dt,
\end{align*}
which together with \eqref{eq:tri83-2} with $r=1$, \eqref{eq:maximal-uniform} and \eqref{eq:FS} yields that the inequality
\begin{equation*}
	K_1
	\lec \norm{(\psi_{k})_{k\in\Z}}_{L^{p'}(\R;\ell^{\infty}(\Z))}\cdot\norm{\pd_x^nv}_{L^{p}(\R;L^1\left(\R_+\right))}
\end{equation*}
holds for any $(\psi_{k}(t))_{k\in\Z}\in L^{p'}(\R;\ell^{\infty}(\Z))$. Thus for $1\leq p<\infty$ and $q=1$, we achieve
\begin{equation*}
	\norm{\left(2^{(\frac{n}{2}-\frac{1}{2q})k} \Delta_{k}^{\mathcal{T}}g\right)_{k\in\Z}}_{L^{p}(\R;\ell^{q}(\Z))}\lec\norm{\pd_x^nv}_{L^{p}(\R;L^q\left(\R_+\right))},
\end{equation*}
which leads to \eqref{eq:trace-1D}.
\end{proof}
\begin{lem}\label{lem:trace-1d-2}
Under the assumptions of Theorem \ref{thm:1D} with $1<p\leq\infty$ and $q= \infty$, estimate \eqref{eq:trace-1D} holds.
\end{lem}
\begin{proof}
By an argument similar to that used for \eqref{eq:lem-trace-2}, we have that for $q=\infty$,
\begin{align}\label{eq:lem-trace-3}
		\left|\Delta_{k}^{\mathcal{T}}g\right|
		&\lec\left|\int_{0}^\infty \cF_t^{-1}\left(x(-\sqrt{i\tau})^{2-n}\cdot e^{-x\sqrt{i\tau}}\cdot\tilde{\theta}(2^{-k}\tau)\cdot\theta(2^{-k}\tau)\cdot \pd_x^n\cF_tv(x,\tau)\right)\dx\right|\notag\\[3pt]
		&\lec\int_{\R}\int_{0}^\infty\frac{  x2^{\frac{4-n}{2}k}\cdot e^{-cx 2^{\frac{k}{2}}}}{\left(1+|2^k s|^2\right)^2}\cdot\left|\Delta_k^{\cT}\pd_x^nv(x,t-s)\right|\dx\ds\notag\\[3pt]
		&\lec\int_{0}^\infty  x2^{\frac{2-n}{2}k}\cdot e^{-cx 2^{\frac{k}{2}}} \left(\sup_{s\in\R}\frac{ \left|\Delta_k^{\cT}\pd_x^nv(x,t-s)\right|}{1+|2^k s|^2}\right)\dx\notag\\[3pt]
		&\lec 2^{-\frac{n}{2}k}\norm{\sup_{s\in\R}\frac{ \left|\Delta_k^{\cT}\pd_x^nv(\cdot,t-s)\right|}{1+|2^k s|^2}}_{L^\infty(\R_+)}.
\end{align}
Applying \eqref{eq:peetre} with $r=1$ and \eqref{eq:maximal-uniform}, we obtain
\begin{equation}\label{Dual2}
\sup_{s\in \R}\frac{|\Delta_k^{\cT}\pd_{x}^nv(x,t-s)|}{1+|2^{k}s|^2}\lesssim M^{\cT}M^{\cT} \pd_x^n v(x,t),
\end{equation}
where $M^{\cT}$ is the Hardy-Littlewood maximal operator acting on the temporal variable $t$. By substituting \eqref{Dual2} into \eqref{eq:lem-trace-3}, we achieve
\begin{equation*}
	\norm{\left(2^{\frac{n}{2}k} \Delta_{k}^{\mathcal{T}}g\right)_{k\in\Z}}_{\ell^{\infty}(\Z)}
	\lec \norm{M^{\cT}M^{\cT} \pd_x^n v(\cdot,t)}_{L^\infty(\R_+)}\lec M^{\cT}M^{\cT} \norm{\pd_x^n v(\cdot,t)}_{L^\infty(\R_+)},
\end{equation*}
which together with \eqref{Maximal1} implies that for $1<p\leq\infty$,
\begin{align*}
	\norm{\left(2^{\frac{n}{2}k} \Delta_{k}^{\mathcal{T}}g\right)_{k\in\Z}}_{L^p(\R;\ell^{\infty}(\Z))}\lec \norm{\pd_x^nv}_{L^{p}(\R;L^\infty\left(\R_+\right))}.
\end{align*}
Hence \eqref{eq:trace-1D} holds for $1<p\leq\infty$ and $q=\infty$.
\end{proof}

\section{The higher-dimensional case}\label{sec:DD}

In this section, we prove Theorems \ref{thm:dd-max-reg} and \ref{thm:dd-trace}.

\subsection{Maximal regularity estimates}
In this subsection, we prove Theorem \ref{thm:dd-max-reg} only in the case $\beta=0$ and $n=|\alpha|$, while for $\beta>0$, the same conclusion follows from the fact that $\pd_t v-\Delta v=0$.
For the caloric extension $v$ defined in \eqref{eq:def-v}, as in \eqref{eq:1D-v}, we have by \cite[Corollary 2.5]{Chen-Liang-Tsai-2025}
\begin{equation}\label{eq:v-g-hat}
    \hat{v}(\xi',x_d,\tau)=  e^{-x_d\sqrt{|\xi'|^2+i\tau}} \cdot \hat{g}(\xi',\tau), 
\end{equation}
and hence
\begin{equation*}
 \pd_x^{\alpha}\Delta_{m}^h\Delta_{k}^{\mathcal{T}}v=\mathcal{F}_{x',t}^{-1}\left(P(\xi',x_d,\tau) \cdot\varphi(2^{-m}\xi')\cdot\theta(2^{-k}\tau)\cdot \hat{g}\right),
\end{equation*}
where
\begin{equation*}
	P(\xi',x_d,\tau)\eqdefa\prod_{j=1}^{d-1}(i\xi_j)^{\alpha_j} \cdot(-\sqrt{|\xi'|^2+i\tau})^{\alpha_d}\cdot e^{-x_d\sqrt{|\xi'|^2+i\tau}}.
\end{equation*}
Now, we decompose 
\begin{equation}\label{eq:vmk1}
	\pd_x^{\alpha}v=\sum_{m,k\in\mathbb{Z}}\pd_x^{\alpha}\Delta_{m}^h\Delta_{k}^{\mathcal{T}}v=I_1+I_2,
\end{equation}
where
\begin{align}\label{eq:vmk2}
	I_1&\eqdefa \sum_{m,k\in\mathbb{Z},k<2m}\pd_x^{\alpha}\Delta_{m}^h\Delta_{k}^{\mathcal{T}}v
	=\sum_{m\in\mathbb{Z}}\pd_x^{\alpha}\Delta_{m}^h\Delta_{<2m}^{\mathcal{T}}v\notag\\[3pt]
	&=\sum_{m\in\mathbb{Z}}\,\frac{1}{(2\pi)^{\frac{d}{2}}}\int_{\R^{d-1}\times\R}\Phi_{m}^{h}(x'-y',x_d,t-s)\cdot \Delta_{m}^h\Delta_{< 2m}^{\cT} g(y',s)\dy'\ds,
\end{align}
with
\begin{equation*}
 \Phi_{m}^{h}(x,t)\eqdefa\mathcal{F}_{x',t}^{-1}\left(P(\xi',x_d,\tau)\cdot\tilde{\varphi}(2^{-m}\xi')\cdot\tilde{\theta}_0(2^{-2m}\tau)\right),
\end{equation*}
and
\begin{align}\label{eq:vmk3}
	I_2&\eqdefa\sum_{m,k\in\mathbb{Z},k\geq 2m}\pd_x^{\alpha}\Delta_{m}^h\Delta_{k}^{\mathcal{T}}v
	=\sum_{k\in\mathbb{Z}}\pd_x^{\alpha}\Delta_{\leq \lfloor k /2\rfloor}^h\Delta_{k}^{\mathcal{T}}v\notag\\[3pt]
	&=\sum_{k\in\mathbb{Z}}\,\frac{1}{(2\pi)^{\frac{d}{2}}}\int_{\R^{d-1}\times\R}\Phi_{k}^{\mathcal{T}}(x'-y',x_d,t-s)\cdot \Delta_{k}^{\mathcal{T}} g(y',s)\dy'\ds,
\end{align}
with
\begin{equation*}
	\Phi_{k}^{\mathcal{T}}(x,t)\eqdefa\mathcal{F}_{x',t}^{-1}\left(P(\xi',x_d,\tau)\cdot\varphi_0(2^{-\lfloor k /2\rfloor}\xi')\cdot\tilde{\theta}(2^{-k}\tau)\right).
\end{equation*}
\begin{lem}\label{lem:exp-decay}
	There exists a constant $c>0$ such that
\begin{equation}\label{eq:est-Phih}
	\norm{\Phi_{m}^{h}(\cdot,x_d,t)}_{L^{1}(\R^{d-1})}\lec\frac{2^{m(2+n)}} { (1+|2^{2m} t|^2)^2} \cdot e^{-cx_d2^{m}},
\end{equation}
and
\begin{equation}\label{eq:est-Phit}
	\norm{\Phi_{k}^{\mathcal{T}}(\cdot,x_d,t)}_{L^{1}(\R^{d-1})}\lec  \frac{2^{\frac{k(2+n)}{2}}} { (1+|2^{k} t|^2)^2} \cdot e^{-cx_d2^{\frac{k}{2}}}.
\end{equation}
\end{lem}
\begin{proof}
The proof is similar to that of \eqref{eq:1d-1}.
By scaling, we have
\begin{align*}
 \Phi_{m}^h(x,t)&=\frac{1}{(2\pi)^{\frac{d}{2}}}\int_{\R^{d-1}\times\R} e^{i(x'\cdot\xi'+t\cdot\tau)}P(\xi',x_d,\tau)\cdot\tilde{\varphi}(2^{-m}\xi')\cdot\tilde{\theta}_0(2^{-2m}\tau)\dxi'\dtau\\[3pt]
	&=\frac{2^{m(d+1)}}{(2\pi)^{\frac{d}{2}}}\int_{\R^{d-1}\times\R} e^{i(2^{m}x'\cdot\xi'+2^{2m}t\cdot\tau)}P(2^{m}\xi',x_d,2^{2m}\tau)\cdot\tilde{\varphi}(\xi')\cdot\tilde{\theta}_0(\tau)\dxi'\dtau.
\end{align*}
By $E\eqdefa e^{i(2^{m}x'\cdot\xi'+2^{2m}t\cdot\tau)} =\left(\frac{1-\Delta_{\xi'}}{1+|2^m x'|^{2}} \right)^d\left(\frac{1-\pd_{\tau}^2}{1+|2^{2m} t|^2}\right)^2E$ and repeated integration by parts,  we obtain
\begin{align*}
 |\Phi_{m}^h(x,t)|&\lec \frac{2^{m(d+1)}}{\left(1+|2^m x'|^{2}\right)^d\cdot\left(1+|2^{2m} t|^2\right)^2}\notag\\[3pt]
&\quad\times \norm{\left(1-\Delta_{\xi'}\right)^d\left(1-\pd_{\tau}^2\right)^2\left(P(2^{m}\xi',2^{2m}\tau)\cdot\tilde{\varphi}(\xi')\cdot\tilde{\theta}_0(\tau)\right)}_{L^{1}(\R^{d-1}\times\R)}.
\end{align*}
Notice that there exists a constant $c>0$ such that
\begin{align*}
\hbox{Re}(\sqrt{|\xi'|^2+i\tau})\geq \frac{1}{\sqrt 2}(|\xi'|^4+|\tau|^2)^{1/4}\geq 2c	
\end{align*}
on  $\supp\tilde{\varphi}\times\supp\tilde{\theta}_0$. Therefore, we infer that
\begin{equation*}
	|\Phi_{m}^h(x,t)| \lec \frac{2^{m(d+1+n)}} {\left(1+|2^m x'|^{2}\right)^d \cdot \left(1+|2^{2m} t|^2\right)^2} \cdot e^{-cx_d2^{m}},
\end{equation*}
which leads to \eqref{eq:est-Phih}. Analogously, we can derive \eqref{eq:est-Phit}.
\end{proof}

\begin{lem}[The estimate of $I_1$]\label{lem:I1}
Under the assumptions of Theorem \ref{thm:dd-max-reg}, for $I_1$ defined in \eqref{eq:vmk2}, we have that
for $n-\frac{1}{q}=0$ $($i.e.~$n=q=1)$,
\begin{equation}\label{I10}
	\norm{I_1}_{L^{p}\big(\R;\,L^{q}(\R_+^{d})\big)}\lec \norm{g}_{L^{p}\big(\R;\,\dotc{B}_{q,q}^{n-\frac{1}{q}}(\R^{d-1})\big)},
\end{equation}
and for $n-\frac{1}{q}<0$ $($i.e. $n=0)$,
\begin{equation}\label{eq:I11}
 \norm{I_1}_{L^{p}\big(\R;\,L^{q}(\R_+^{d})\big)}\lec \norm{g}_{\dotc {F}^{\frac{n}{2}-\frac{1}{2q}}_{p,q}\big(\R;\,L^{q}(\R^{d-1})\big)},
\end{equation}
and for $n-\frac{1}{q}>0$,
\begin{equation}\label{eq:I12}
	\norm{I_1}_{L^{p}\big(\R;\,L^{q}(\R_+^{d})\big)}\lec \norm{g}_{\dotc {F}^{\frac{n}{2}-\frac{1}{2q}}_{p,q}\big(\R;\,L^{q}(\R^{d-1})\big)}+\norm{g}_{L^{p}\big(\R;\,\dotc{B}_{q,q}^{n-\frac{1}{q}}(\R^{d-1})\big)}.
\end{equation}
\end{lem}

\begin{proof}
We first deal with the case $n-\frac{1}{q}=0$ (i.e.~$n=q=1$) by a duality argument. 
As in \eqref{Dual1}, we get that for any function $\psi(x,t)\in L^{p'}\big(\R;\,L^{\infty}(\R_+^{d})\big)$,
\begin{align*}
	K_2&\eqdefa \left|\int_{\R^d_+\times\R}I_1\cdot\psi\dx\dt\right|\notag\\[3pt]
	&\lec\left|\sum_{m\in\mathbb{Z}}\int_{\R^d_+\times\R}\left(\int_{\R^{d-1}\times\R}\Phi_{m}^{h}(x'-y',x_d,t-s)\cdot \Delta_{m}^h g(y',s)\dy'\ds\right) \cdot\Delta_{< 2m}^{\cT}\psi\dx\dt\right|,
\end{align*}	
which together with \eqref{eq:est-Phih} yields (denoting $s=t'$, $t-s=-s'$, and dropping prime)
\begin{align}
K_2&\lec\sum_{m\in\mathbb{Z}}\int_{\R_+\times\R}2^{m} e^{-cx_d2^{m}}\cdot\norm{\Delta_{m}^h g}_{L^{1}(\R^{d-1})}\cdot \sup_{s\in\R}\frac{\norm{\Delta_{< 2m}^{\cT}\psi(\cdot,x_d,t-s)}_{L^{\infty}(\R^{d-1})}} {1+|2^{2m} s|^2}\dx_d\dt\notag\\[3pt]
&\lec \sum_{m\in\mathbb{Z}}\int_{\R}\norm{\Delta_{m}^h g}_{L^{1}(\R^{d-1})}\cdot\sup_{s\in\R}\frac{\norm{\Delta_{< 2m}^{\cT}\psi(\cdot,t-s)}_{L^{\infty}(\R^{d}_+)}} {1+|2^{2m} s|^2}\dt.\label{Dual3}
\end{align}
By \eqref{eq:maximal-uniform} and \eqref{eq:peetre} with $r=1$, we have
\begin{equation}\label{Dual4}
\sup_{s\in\R}\frac{\norm{\Delta_{< 2m}^{\cT}\psi(\cdot,t-s)}_{L^{\infty}(\R^{d}_+)}} {1+|2^{2m} s|^2}\lec M^{\cT}M^{\cT}\norm{\psi(\cdot,t)}_{L^{\infty}(\R^{d}_+)}.
\end{equation}
By substituting \eqref{Dual4} into \eqref{Dual3} and then using \eqref{Maximal1}, we achieve that for $1\leq p<\infty$,
\begin{align*}
K_2&\lec \norm{g}_{L^{p}\big(\R;\,\dotc{B}_{1,1}^{0}(\R^{d-1})\big)}\cdot\norm{M^{\cT}M^{\cT} \norm{\psi(\cdot,t)}_{L^{\infty}(\R^{d}_+)}}_{L^{p'}(\R)}\\[3pt]
&\lec \norm{g}_{L^{p}\big(\R;\,\dotc{B}_{1,1}^{0}(\R^{d-1})\big)}\cdot\norm{\psi}_{L^{p'}\big(\R;\,L^{\infty}(\R_+^{d})\big)},
\end{align*}
which leads to \eqref{I10} by the standard norm characterization
\[
\|u\|_{L^p(\mathbb R;L^q(\mathbb R_+^d))}
=
\sup_{\|\psi\|_{L^{p'}(\mathbb R;L^{q'}(\mathbb R_+^d))}\leq 1}
\left|
\int_{\mathbb R_+^d\times\mathbb R}
u\cdot\psi\,dx\,dt
\right|, \quad 1\leq p,q\leq \infty.
\]

Now we assume $n-\frac{1}{q}\neq0$.	By \eqref{eq:vmk2}, \eqref{eq:est-Phih} and Young's inequality, we obtain
\begin{align}
\norm{I_1(\cdot,x_d,t)}_{L^q(\R^{d-1})}&\lec \sum_{m\in\mathbb{Z}}\int_{\R}\norm{\Phi_{m}^{h}(\cdot,x_d,s)}_{L^{1}(\R^{d-1})}\cdot\norm{\Delta_{m}^{h}\Delta_{< 2m}^{\cT}g(\cdot,t-s)}_{L^{q}(\R^{d-1})} \ds\notag\\[3pt]
&\lec \sum_{m\in\mathbb{Z}}2^{mn} e^{-cx_d2^{m}} \cdot \sup_{s\in\R}\frac{\norm{\Delta_{m}^{h}\Delta_{< 2m}^{\cT}g(\cdot,t-s)}_{L^{q}(\R^{d-1})}} {1+|2^{2m} s|^2}.\label{eq:I_1-est1}
\end{align}
The above sup is bounded for each $m$, since $g\in L^{p}\big(\R;\,\dotc{B}_{q,q}^{n-\frac{1}{q}}(\R^{d-1})\big)$ if $n-\frac{1}{q}>0$ and $g\in \dotc {F}^{\frac{n}{2}-\frac{1}{2q}}_{p,q}\big(\R;\,L^{q}(\R^{d-1})\big)$ if $n-\frac{1}{q}<0$. Indeed, if $n-\frac{1}{q}>0$, by Young's inequality,
	\begin{equation*}
		\norm{\norm{\Delta^{h}_{m}\Delta^{\cT}_{<2m} g}_{L^q(\R^{d-1})}}_{L^{\infty}(\R)}\lec_{m} \norm{\norm{\Delta^{h}_{m} g}_{L^q(\R^{d-1})}}_{L^{p}(\R)}\lec_{m}\norm{g}_{L^{p}\big(\R;\,\dotc{B}_{q,q}^{n-\frac{1}{q}}(\R^{d-1})\big)},
	\end{equation*}
and if $n-\frac{1}{q}<0$, applying Young's inequality and H\"older's inequality in $\ell^1$ gives
\begin{align*}
\norm{\norm{\Delta^{h}_{m}\Delta^{\cT}_{<2m} g}_{L^q(\R^{d-1})}}_{L^{\infty}(\R)}&\lec_{m} \norm{\sum_{j<2m} 2^{-(\frac{n}{2}-\frac{1}{2q})j}\cdot 2^{(\frac{n}{2}-\frac{1}{2q})j}\norm{\Delta^{\cT}_{j} g}_{L^q(\R^{d-1})}}_{L^{p}(\R)}\\[3pt]
&\lec_{m}\norm{g}_{\dotc {F}^{\frac{n}{2}-\frac{1}{2q}}_{p,q}\big(\R;\,L^{q}(\R^{d-1})\big)}.
\end{align*}
Estimate \eqref{eq:I_1-est1},
together with \eqref{eq:tri83-2} with $r=\frac12$ and \eqref{eq:int-xd1},  implies that
\begin{align}\label{eq:lem-I1}
\norm{I_1}_{L^{p}\big(\R;\,L^{q}(\R_+^{d})\big)}&= \norm{\left(\int_0^{\infty}\norm{I_1(\cdot,x_d,t)}_{L^{q}(\R^{d-1})}^{q}\dx_d\right)^{\frac{1}{q}}}_{L^p(\R)}\notag\\[3pt]
&\lec \norm{\left(\sum_{m\in\mathbb{Z}}2^{mnq-m} \Big(\sup_{s\in\R}\frac{\norm{\Delta_{m}^{h}\Delta_{< 2m}^{\cT}g(\cdot,t-s)}_{L^{q}(\R^{d-1})}} {1+|2^{2m} s|^2}\Big)^q\right)^{\frac{1}{q}}}_{L^p(\R)}\notag\\[3pt]
&\lec \norm{\left(2^{m(n-\frac{1}{q})}\norm{\Delta^{h}_{m}\Delta^{\cT}_{<2m} g}_{L^q(\R^{d-1})}\right)_{m\in\Z}}_{L^{p}(\R;\ell^q(\Z))}.
\end{align}
In what follows, we shall bound the RHS of \eqref{eq:lem-I1} separately in the case $n-\frac{1}{q}<0$ and $n-\frac{1}{q}>0$.

\textit{Case 1.}
If $n-\frac{1}{q}<0$, we get
\begin{align*}
2^{m(n-\frac{1}{q})}\norm{\Delta^{h}_{m}\Delta^{\cT}_{<2m} g}_{L^q(\R^{d-1})}&\lec 2^{m(n-\frac{1}{q})}\norm{\Delta^{\cT}_{<2m}g}_{L^q(\R^{d-1})}\notag\\[3pt]
&\lec\sum_{j<2m}2^{(m-\frac{j}{2})(n-\frac{1}{q})}\cdot2^{\frac{j}{2}(n-\frac{1}{q})}\norm{\Delta^{\cT}_{j}g}_{L^q(\R^{d-1})},
\end{align*}
which together with \eqref{eq:lem-I1} and Young's convolution inequality for $\ell^q$ implies that
\begin{equation*}
\norm{I_1}_{L^{p}\big(\R;\,L^{q}(\R_+^{d})\big)}
\lec\norm{\left(2^{\frac{j}{2}(n-\frac{1}{q})}\norm{\Delta^{\cT}_j g}_{L^q(\R^{d-1})}\right)_{j\in\Z}}_{L^{p}(\R;\ell^q(\Z))}.
\end{equation*}
Hence \eqref{eq:I11}.

\textit{Case 2.}
If $n-\frac{1}{q}>0$, we get
\begin{align*}
&\quad\ 2^{m(n-\frac{1}{q})}\norm{\Delta^{h}_{m}\Delta^{\cT}_{<2m} g}_{L^q(\R^{d-1})}\notag\\[3pt]
&\lec 2^{m(n-\frac{1}{q})}\Big(\norm{\Delta^{h}_{m}g}_{L^q(\R^{d-1})}+\norm{\Delta^{h}_{m}\Delta^{\cT}_{\geq2m}g}_{L^q(\R^{d-1})}\Big)\notag\\[3pt]
&\lec 2^{m(n-\frac{1}{q})}\Big(\norm{\Delta^{h}_{m}g}_{L^q(\R^{d-1})}+\norm{\Delta^{\cT}_{\geq2m}g}_{L^q(\R^{d-1})}\Big)\notag\\[3pt]
&\lec2^{m(n-\frac{1}{q})}\norm{\Delta^{h}_{m}g}_{L^q(\R^{d-1})}+\sum_{j\geq2m}2^{(m-\frac{j}{2})(n-\frac{1}{q})}\cdot2^{\frac{j}{2}(n-\frac{1}{q})}\norm{\Delta^{\cT}_{j}g}_{L^q(\R^{d-1})},
\end{align*}
which together with \eqref{eq:lem-I1} and Young's convolution inequality for $\ell^q$ implies that
\begin{align*}
\norm{I_1}_{L^{p}\big(\R;\,L^{q}(\R_+^{d})\big)}
&\lec\norm{\left(2^{m(n-\frac{1}{q})}\norm{\Delta^{h}_m g}_{L^q(\R^{d-1})}\right)_{m\in\Z}}_{L^{p}(\R;\ell^q(\Z))}\notag\\[3pt]
&\quad\ +\norm{\left(2^{\frac{j}{2}(n-\frac{1}{q})}\norm{\Delta^{\cT}_j g}_{L^q(\R^{d-1})}\right)_{j\in\Z}}_{L^{p}(\R;\ell^q(\Z))}.
\end{align*}
Hence \eqref{eq:I12}.
\end{proof}

\begin{lem}[The estimate of $I_2$] \label{lem:I2}
Under the assumptions of Theorem \ref{thm:dd-max-reg}, for $I_2$ defined in \eqref{eq:vmk3}, we have 
\begin{equation}\label{eq:I2}
\norm{I_2}_{L^{p}\big(\R;\,L^{q}(\R_+^{d})\big)}\lec \norm{g}_{\dotc {F}^{\frac{n}{2}-\frac{1}{2q}}_{p,q}\big(\R;\,L^{q}(\R^{d-1})\big)}.
\end{equation}
\end{lem}
\begin{proof}
By \eqref{eq:vmk3}, \eqref{eq:est-Phit} and Young's inequality, we obtain
\begin{align*}
\norm{I_2(\cdot,x_d,t)}_{L^q(\R^{d-1})}&\lec \sum_{k\in\mathbb{Z}}\int_{\R}\norm{\Phi_{k}^{\cT}(\cdot,x_d,s)}_{L^{1}(\R^{d-1})}\cdot\norm{\Delta_{k}^{\cT}g(\cdot,t-s)}_{L^{q}(\R^{d-1})} \ds\\[3pt]
&\lec \sum_{k\in\mathbb{Z}}2^{\frac{kn}{2}} e^{-cx_d2^{\frac{k}{2}}} \cdot \sup_{s\in\R}\frac{\norm{\Delta_{k}^{\cT}g(\cdot,t-s)}_{L^{q}(\R^{d-1})}} {1+|2^{k} s|^2}.
\end{align*}
Since $g\in\dotc {F}^{\frac{n}{2}-\frac{1}{2q}}_{p,q}\big(\R;\,L^{q}(\R^{d-1})\big)$, the above sup is 
bounded for each $k$. The above estimate
together with \eqref{eq:tri83-2} with $r=\frac12$ and \eqref{eq:int-xd1}  implies that 
\begin{align*}
&\,\quad \norm{\left(\int_0^{\infty}\norm{I_2(\cdot,x_d,t)}_{L^{q}(\R^{d-1})}^{q}\dx_d\right)^{\frac{1}{q}}}_{L^p(\R)}\\[3pt]
&\lec \norm{\left(\sum_{k\in\mathbb{Z}}2^{\frac{knq-k}{2}} \Big(\sup_{s\in\R}\frac{\norm{\Delta_{k}^{\cT}g(\cdot,t-s)}_{L^{q}(\R^{d-1})}} {1+|2^{k} s|^2}\Big)^q\right)^{\frac{1}{q}}}_{L^p(\R)}\\[3pt]
&\lec \norm{\left(2^{k(\frac{n}{2}-\frac{1}{2q})}\norm{\Delta^{\cT}_{k} g}_{L^q(\R^{d-1})}\right)_{k\in\Z}}_{L^{p}(\R;\ell^q(\Z))}.
\end{align*}
Hence \eqref{eq:I2}.
\end{proof}

We are now in a position to prove Theorem \ref{thm:dd-max-reg}.

\begin{proof}[Proof of Theorem \ref{thm:dd-max-reg}]
By \eqref{eq:vmk1}, Lemma \ref{lem:I1}, and Lemma \ref{lem:I2}, we obtain \eqref{eq:regu-est1} and \eqref{eq:regu-est2}. This completes the proof.
\end{proof}

\subsection{Caloric trace estimates}

In this subsection, we present the proof of Theorem \ref{thm:dd-trace}, which follows from Lemmas \ref{lem:trace-d0} and \ref{lem:trace-d0b}, bounding the two terms on the left-hand side of \eqref{thm:dd-trace-eq1} separately.

\begin{lem}\label{lem:trace-d0}
Let $d\geq 2$. Suppose that  $g\in \mathcal{S}(\R^{d-1}\times \R)$ and let $v(x,t)$ be defined by \eqref{eq:def-v}. Then, for $1<p<\infty$ and $1\leq q\leq \infty$, as well as for $p=q=1$ and $p=q=\infty$, and for every integer $n\geq 0$, we have
\begin{equation}\label{eq:trace1-1}
	\norm{\left(2^{(\frac{n}{2}-\frac{1}{2q})k}\|\Delta_{\leq \frac{k}{2}}^{h}\Delta_{k}^{\cT} g\|_{L^{q}(\R^{d-1})}\right)_{k\in\Z}}_{L^p(\R;\ell^q(\Z))}
		\lec \norm{\pd_{x_d}^nv}_{L^{p}\big(\R;\,L^{q}(\R_+^{d})\big)}.
\end{equation}
Moreover, if $n>\frac{1}{q}$, then
	\begin{equation}\label{eq:trace1-2}
		\norm{\left(2^{(\frac{n}{2}-\frac{1}{2q})k}\|\Delta_{>\frac{k}{2}}^{h}\Delta_{k}^{\cT} g\|_{L^{q}(\R^{d-1})}\right)_{k\in\Z}}_{L^p(\R;\ell^q(\Z))}
		\lec \norm{\pd_{x_d}^nv}_{L^{p}\big(\R;\,L^{q}(\R_+^{d})\big)}.
	\end{equation}
\end{lem}
As in the one-dimensional case, we divide the proof into the following three lemmas (Lemmas \ref{lem:trace-d1}-\ref{lem:trace-d3}) corresponding to different ranges of the indices $p,q$.

\begin{lem}\label{lem:trace-d1}
Under the assumptions of Lemma \ref{lem:trace-d0} with $1<p,q<\infty$, estimates \eqref{eq:trace1-1} and \eqref{eq:trace1-2} hold.
\end{lem}
\begin{proof} Let $\gamma=\lfloor \frac 1q\rfloor\in\{0,1\}$ for $1\leq q \leq \infty$ ($\gamma=0$ here and in Lemma \ref{lem:trace-d3}, $\gamma=1$ in Lemma \ref{lem:trace-d2}).   In view of \eqref{eq:v-g-hat}, we have
\begin{equation}\label{eq:lem:trace-d}
		\hat{g}(\xi',\tau)=(-2)^{1+\gamma}\int_{0}^{+\infty}x_d^{\gamma} (-\sqrt{|\xi'|^2+i\tau})^{1+\gamma-n}\cdot e^{-x_d\sqrt{|\xi'|^2+i\tau}}\cdot \pd_{x_d}^n\hat{v}(\xi',x_d,\tau) \dx_d.
\end{equation}

We now use a duality argument. 
For any sequence $\{\psi_k(x',t)\}_{k\in\mathbb Z}$ satisfying
\begin{equation*}
	\norm{\left(\norm{\psi_k(\cdot,t)}_{L^{q'}(\mathbb R^{d-1})}\right)_{k\in\Z}}_{L^{p'}(\mathbb R;\ell^{q'}(\Z))}<\infty,
\end{equation*}
we define
\begin{align}
K_{31}&\eqdefa\sum_{k\in\mathbb{Z}}2^{(\frac{n}{2}-\frac{1}{2q})k}\int_{\R^{d-1}\times\R} \Delta_{\leq \frac{k}{2}}^{h}\Delta_{k}^{\mathcal{T}}g\cdot\psi_{k}\dx'\dt\notag\\[3pt]
&=\sum_{k\in\mathbb{Z}}2^{(\frac{n}{2}-\frac{1}{2q})k}\int_{\R^{d-1}\times\R} \Delta_{\leq \frac{k}{2}}^{h}\Delta_{k}^{\mathcal{T}}g\cdot\tilde{\Delta}_{k}^{\mathcal{T}}\psi_{k}\dx'\dt,\label{K31}
\end{align}
and
\begin{align}
K_{32}&\eqdefa\sum_{k\in\mathbb{Z}}2^{(\frac{n}{2}-\frac{1}{2q})k}\int_{\R^{d-1}\times\R} \Delta_{>\frac{k}{2}}^{h}\Delta_{k}^{\mathcal{T}}g\cdot\psi_{k}\dx'\dt\notag\\[3pt]
&=\sum_{k\in\mathbb{Z}}2^{(\frac{n}{2}-\frac{1}{2q})k}\int_{\R^{d-1}\times\R} \sum_{j>\frac{k}{2}}\Delta_{j}^{h}\Delta_{k}^{\mathcal{T}}g\cdot\tilde{\Delta}_{k}^{\mathcal{T}}\psi_{k}\dx'\dt.\label{K32}
\end{align}

We estimate $K_{31}$ first. By an argument similar to that used in Lemma \ref{lem:exp-decay},
we have
\begin{align*}
 &\quad\ \norm{\cF_{x',t}^{-1}\Big(x_d^{\gamma}(-\sqrt{|\xi'|^2+i\tau})^{1+\gamma-n}\cdot e^{-x_d\sqrt{|\xi'|^2+i\tau}}\cdot\theta(2^{-k}\tau)\cdot \varphi_0(2^{-\lfloor \frac{k}{2}\rfloor }\xi')\Big)}_{L^{1}(\R^{d-1})}\\[3pt]
 &\lec\frac{x_d^{\gamma}2^{\frac{3+\gamma-n}{2}k}}{\left(1+|2^{k} t|^2\right)^2}\cdot e^{-cx_d 2^{\frac{k}{2}}},
\end{align*}
which together with \eqref{eq:lem:trace-d}, Minkowski's and Young's inequalities implies that
\begin{equation}\label{eq:hg1}
	\|\Delta_{\leq \frac{k}{2}}^{h}\Delta_{k}^{\cT} g\|_{L^{q}(\R^{d-1})}\lec\int_{0}^{\infty}\int_{\R} \frac{ x_d^{\gamma}2^{\frac{3+\gamma-n}{2}k} e^{-cx_d 2^{\frac{k}{2}}}}{\left(1+|2^k (t-s)|^2\right)^2}\norm{\pd_{x_d}^nv(\cdot,x_d,s)}_{L^{q}(\R^{d-1})}\ds\dx_d.
\end{equation}
As in the proof of Lemma \ref{lem:trace-1d-1}, we get that, by \eqref{eq:hg1},
\begin{align}
	K_{31}&\lec\sum_{k\in\mathbb{Z}}\int_{\R}\Big(\int_{0}^{\infty}\int_{\R}\frac{ x_d^{\gamma}2^{\frac{3+\gamma-\frac{1}{q}}{2}k} e^{-cx_d 2^{\frac{k}{2}}}}{\left(1+|2^k (t-s)|^2\right)^2}\cdot\norm{\pd_{x_d}^nv(\cdot,x_d,s)}_{L^{q}(\R^{d-1})}\ds\dx_d\Big)\notag\\[3pt]
	&\quad\quad \times\norm{\tilde{\Delta}_{k}^{\mathcal{T}}\psi_{k}(\cdot,t)}_{L^{q'}(\R^{d-1})}\dt\notag \\[3pt]
	&\lec\int_{\R}\int_{0}^{\infty}\sum_{k\in\mathbb{Z}}\Big( x_d^{\gamma}2^{\frac{1+\gamma-\frac{1}{q}}{2}k} e^{-cx_d 2^{\frac{k}{2}}}\cdot\Big(\sup_{s\in \R} \frac{\norm{\tilde{\Delta}_{k}^{\mathcal{T}}\psi_{k}(\cdot,t-s)}_{L^{q'}(\R^{d-1})}}{1+|2^ks|^2}\Big)\Big)\notag\\[3pt]
	&\quad\quad
	\times\norm{\pd_{x_d}^nv(\cdot,x_d,t)}_{L^{q}(\R^{d-1})}\dx_d\dt.\label{eq:K31}
\end{align}
Applying H\"older's inequality and then \eqref{eq:int-xd1}, we obtain
\begin{equation}\label{eq:K31-bound}
	K_{31}\lec\int_{\R}\Big(\sum_{k\in\mathbb{Z}}\Big(\sup_{s\in \R} \frac{\norm{\tilde{\Delta}_{k}^{\mathcal{T}}\psi_{k}(\cdot,t-s)}_{L^{q'}(\R^{d-1})}}{1+|2^ks|^2}\Big)^{q'}\Big)^{\frac{1}{q'}}\cdot\norm{\pd_{x_d}^nv(\cdot,t)}_{L^{q}(\R_+^{d})}\dt.
\end{equation}
By \eqref{eq:peetre} with $r=1$, Minkowski inequality and \eqref{eq:maximal-uniform}, we have
\begin{equation}\label{4.22}
\sup_{s\in \R} \frac{\norm{\tilde{\Delta}_{k}^{\mathcal{T}}\psi_{k}(\cdot,t-s)}_{L^{q'}(\R^{d-1})}}{1+|2^ks|^2}
\lec M^{\cT}M^{\cT}\norm{\psi_k(\cdot,t)}_{L^{q'}(\R^{d-1})}.
\end{equation}
Hence by \eqref{eq:FS},
\begin{equation*}
	K_{31}\lec\norm{\left(\norm{\psi_{k}(\cdot,t)}_{L^{q'}(\R^{d-1})}\right)_{k\in\Z}}_{L^{p'}(\R;\ell^{q'}(\Z))}\cdot\norm{\pd_{x_d}^nv}_{L^{p}(\R;L^q\left(\R_+^d\right))}.
\end{equation*}
By a duality argument, it follows that
for $1<p<\infty$ and $1< q<\infty$, estimate \eqref{eq:trace1-1} holds.

We now estimate $K_{32}$. For $j>\frac{k}{2}$, we have 
\begin{align*}
 &\quad\ \norm{\cF_{x',t}^{-1}\Big(x_d^{\gamma}(-\sqrt{|\xi'|^2+i\tau})^{1+\gamma-n}\cdot e^{-x_d\sqrt{|\xi'|^2+i\tau}}\cdot\theta(2^{-k}\tau)\cdot \varphi(2^{-j}\xi')\Big)}_{L^{1}(\R^{d-1})}\\[3pt]
 &\lec \bigg\|\frac{2^{(d-1)j+k}\cdot  x_d^{\gamma}2^{j(1+\gamma-n)}}{(1+|2^jx'|^2)^d(1+|2^kt|^2)^2}\norm{\left(1-\Delta_{\xi'}\right)^d\left(1-\pd_{\tau}^2\right)^2 P_1(\xi',x_d,\tau)}_{L^{1}(\R^{d-1}\times\R)}\bigg\|_{L^1(\R^{d-1})} \\[3pt]
 &\lec\frac{2^kx_d^{\gamma}2^{j(1+\gamma-n)}}{\left(1+|2^{k} t|^2\right)^2}\cdot e^{-cx_d 2^j},
\end{align*}
where 
\[P_1(\xi',x_d,\tau)=\Big(\sqrt{|\xi'|^2+2^{k-2j}i\tau }\Big)^{1+\gamma-n}\cdot e^{-x_d\cdot 2^j\sqrt{|\xi'|^2+2^{k-2j}i\tau}}\cdot\theta(\tau)\cdot \varphi(\xi').\]
The estimate above, together with \eqref{eq:lem:trace-d}, Minkowski's and Young's inequalities, implies that for $j>\frac{k}{2}$,
\begin{equation}\label{eq:hg2}
	\|\Delta_{j}^{h}\Delta_{k}^{\cT} g\|_{L^{q}(\R^{d-1})}\lec\int_{0}^{\infty} \!\!\int_{\R}\frac{ 2^kx_d^{\gamma}2^{j(1+\gamma-n)}e^{-cx_d 2^j}}{\left(1+|2^k (t-s)|^2\right)^2}\norm{\pd_{x_d}^nv(\cdot,x_d,s)}_{L^{q}(\R^{d-1})}\ds\dx_d.
\end{equation}
Given  $1+\gamma-n\leq0$ (as $n>1/q$),  we have that for $\ep>0$,
\begin{equation}\label{eq:gamma-q-n}
	\sum_{j>\frac{k}{2}} 2^{(1+\gamma-n)j}e^{-cx_d 2^j}\lec 2^{(1+\gamma-n)\frac{k}{2}}e^{-\frac{c}{2} x_d2^{\frac{k}{2}}}\sum_{j>\frac{k}{2}}(x_d2^j)^{-\ep}\lec 2^{(1+\gamma-n-\ep)\frac{k}{2}}x_d^{-\ep}e^{-\frac{c}{2}x_d2^{\frac{k}{2}}}.
\end{equation}
Hence we obtain 
\begin{align}
	K_{32}&\lec\sum_{k\in\mathbb{Z}}\int_{\R}\Big(\int_{0}^{\infty}\int_{\R}
	\frac{2^{\frac{3+\gamma-\frac{1}{q}-\ep}{2}k} x_d^{\gamma-\ep}e^{-\frac{c}{2}x_d 2^{\frac{k}{2}}}}{\left(1+|2^k (t-s)|^2\right)^2}\cdot\norm{\pd_{x_d}^nv(\cdot,x_d,s)}_{L^{q}(\R^{d-1})}\ds\dx_d\Big) \notag\\[3pt]
	&\quad\quad \times\norm{\tilde{\Delta}_{k}^{\mathcal{T}}\psi_{k}(\cdot,t)}_{L^{q'}(\R^{d-1})}\dt \notag\\[3pt]
	&\lec\int_{\R}\int_{0}^{\infty}\sum_{k\in\mathbb{Z}}\Big( 2^{\frac{1+\gamma-\frac{1}{q}-\ep}{2}k}x_d^{\gamma-\ep} e^{-\frac{c}{2}x_d 2^{\frac{k}{2}}}\cdot\Big(\sup_{s\in \R} \frac{\norm{\tilde{\Delta}_{k}^{\mathcal{T}}\psi_{k}(\cdot,t-s)}_{L^{q'}(\R^{d-1})}}{1+|2^ks|^2}\Big)\Big)\notag\\[3pt]
	&\quad\quad
	\times\norm{\pd_{x_d}^nv(\cdot,x_d,t)}_{L^{q}(\R^{d-1})}\dx_d\dt.
	\label{eq:K32}
\end{align}
By applying H\"older's inequality and then \eqref{eq:int-xd1} with $\ep=\frac{1}{2q'^2}$, we obtain the same bound as in \eqref{eq:K31-bound} for $K_{32}$. The estimate \eqref{eq:trace1-2} follows in the same way.
\end{proof}

Note that in the previous lemma, we have $\gamma=0$, and \eqref{eq:gamma-q-n} requires $n\geq 1$. In the next lemma for the endpoint case $q=1$, we will choose $\gamma=1$, which produces the bounded term
\[\sum_{k\in \mathbb{Z}}\left((x_d2^{\frac{k}{2}})^\frac{1}{2}e^{-\frac{c}{2}x_d 2^{\frac{k}{2}}}\right).\]
If we instead choose $\gamma=0$, then the corresponding term carries a non-positive power of $x_d$, and the summation diverges. This also forces the condition $n\geq 2$, since we need \eqref{eq:gamma-q-n}. This explains the condition $n>\frac{1}{q}$ in Lemma \ref{lem:trace-d0}.
 
\begin{lem}\label{lem:trace-d2}
Under the assumptions of Lemma \ref{lem:trace-d0} with $1\leq p<\infty$ and $q=1$, estimates \eqref{eq:trace1-1} and \eqref{eq:trace1-2} hold.
\end{lem}
\begin{proof}
The proof is similar to that of Lemma \ref{lem32}. We continue to use \eqref{K31} and \eqref{K32} from Lemma \ref{lem:trace-d1}. By \eqref{eq:K31} with $q=1$,
\begin{equation}
	K_{31}\lec\int_{\R}\sup_{k}\sup_{s\in \R} \frac{\norm{\tilde{\Delta}_{k}^{\mathcal{T}}\psi_{k}(\cdot,t-s)}_{L^{\infty}(\R^{d-1})}}{1+|2^ks|^2}\cdot\norm{\pd_{x_d}^nv(\cdot,t)}_{L^{1}(\R_+^{d})}\dt.  \label{eq:K31-q1}
\end{equation}
Hence by \eqref{4.22} and \eqref{eq:FS}, we derive
\begin{equation*}
	K_{31}\lec\norm{\left(\norm{\psi_{k}(\cdot,t)}_{L^{\infty}(\R^{d-1})}\right)_{k\in\Z}}_{L^{p'}(\R;\ell^{\infty}(\Z))}\cdot\norm{\pd_{x_d}^nv}_{L^{p}(\R;L^1\left(\R_+^d\right))}.
\end{equation*}
Thus by a duality argument, for $1\leq p<\infty$ and $q=1$, we obtain \eqref{eq:trace1-1}.

For $K_{32}$ with $n>1/q=1$,
by \eqref{eq:K32} with  $\ep=\frac{1}{2}$, we get 
\begin{align*}
	K_{32}
	&\lec \int_{\R}\int_0^\infty \sum_{k\in \mathbb{Z}}\left((x_d2^{\frac{k}{2}})^\frac{1}{2}e^{-\frac{c}{2}x_d 2^{\frac{k}{2}}}\right)
	\sup_{k}\sup_{s\in \R} \frac{\norm{\tilde{\Delta}_{k}^{\mathcal{T}}\psi_{k}(\cdot,t-s)}_{L^{\infty}(\R^{d-1})}}{1+|2^ks|^2}\\[3pt]
	&\quad \quad \times \norm{\pd_{x_d}^nv(\cdot,x_d,t)}_{L^{1}(\R^{d-1})}\dx_d\dt.
\end{align*}
Applying \eqref{eq:lem-xd4}, we obtain the same bound \eqref{eq:K31-q1} for $K_{32}$. The estimate \eqref{eq:trace1-2} follows in the same way.
\end{proof}
\begin{lem}\label{lem:trace-d3}
	Under the assumptions of Lemma \ref{lem:trace-d0} with $1< p\leq\infty$ and $q=\infty$, estimates \eqref{eq:trace1-1} and \eqref{eq:trace1-2} hold.
\end{lem}
\begin{proof}
The proof is similar to that of Lemma \ref{lem:trace-1d-2}: we insert the factor $\tilde \theta(2^{-k}\tau)\theta(2^{-k}\tau)$ before $\pd_{x_d}^n\hat{v}$ when we apply $\Delta_k^\cT$ to \eqref{eq:lem:trace-d}. By \eqref{eq:hg1}, \eqref{eq:lem-xd5},  \eqref{eq:peetre} with $r=1$ and \eqref{eq:maximal-uniform}, we have that for $q=\infty$,
\begin{align}
\|\Delta_{\leq \frac{k}{2}}^h\Delta_{k}^{\cT} g\|_{L^{\infty}(\R^{d-1})}&\lec\int_{0}^{\infty}\int_{\R} \frac{ 2^{(\frac{3}{2}-\frac{n}{2})k} e^{-cx_d 2^{\frac{k}{2}}}}{\left(1+|2^k (t-s)|^2\right)^2}\norm{\Delta_{k}^{\mathcal{T}}\pd_{x_d}^nv(\cdot,x_d,s)}_{L^{\infty}(\R^{d-1})}\ds\dx_d\notag \\[3pt]
		&\lec\int_{0}^{\infty} 2^{(\frac{1}{2}-\frac{n}{2})k} e^{-cx_d 2^{\frac{k}{2}}}\cdot\sup_{s\in \R} \frac{\norm{\Delta_{k}^{\mathcal{T}}\pd_{x_d}^nv(\cdot,x_d,t-s)}_{L^{\infty}(\R^{d-1})}}{1+|2^ks|^2}\dx_d \notag\\[3pt]
		&\lec2^{-\frac{n}{2}k}\sup_{s\in \R} \frac{\norm{\Delta_{k}^{\mathcal{T}}\pd_{x_d}^nv(\cdot,t-s)}_{L^{\infty}(\R_+^{d})}}{1+|2^ks|^2} \notag
\\[3pt]
&\lec2^{-\frac{n}{2}k} \cdot M^{\cT}M^{\cT} \norm{\pd_{x_d}^nv(\cdot,t)}_{L^{\infty}(\R_+^{d})}. \label{4.24}
\end{align}
Then, applying \eqref{eq:FS}, we obtain \eqref{eq:trace1-1}.

Analogously, by \eqref{eq:hg2} and \eqref{eq:gamma-q-n} with $n>1/q=0$ and $\ep=\frac{1}{2}$, we have
\begin{align*}
	&\quad\ \sum_{j>\frac{k}{2}}\|\Delta_{j}^h\Delta_{k}^{\cT} g\|_{L^{\infty}(\R^{d-1})}\\[3pt]
	&\lec\int_{0}^{\infty}\int_{\R} \frac{  x_d^{-\frac12}2^{(\frac{5}{2}-n)\frac{k}{2}}e^{-\frac{c}{2}x_d 2^{\frac{k}{2}}}}{\left(1+|2^k (t-s)|^2\right)^2}\norm{\Delta_{k}^{\mathcal{T}}\pd_{x_d}^nv(\cdot,x_d,s)}_{L^{\infty}(\R^{d-1})}\ds\dx_d \notag\\[3pt]
	&\lec\int_{0}^{\infty}x_d^{-\frac12} 2^{(\frac{1}{2}-n)\frac{k}{2}} e^{-\frac{c}{2}x_d 2^{\frac{k}{2}}}\cdot\sup_{s\in \R} \frac{\norm{\Delta_{k}^{\mathcal{T}}\pd_{x_d}^nv(\cdot,x_d,t-s)}_{L^{\infty}(\R^{d-1})}}{1+|2^ks|^2}\dx_d.
\end{align*}
Therefore, by \eqref{eq:lem-xd5},  \eqref{eq:peetre} with $r=1$ and \eqref{eq:maximal-uniform}, we achieve
\begin{align*}
\|\Delta_{>\frac{k}{2}}^{h}\Delta_{k}^{\cT} g\|_{L^{\infty}(\R^{d-1})} 
&\lec2^{-\frac{n}{2}k}\sup_{s\in \R} \frac{\norm{\Delta_{k}^{\mathcal{T}}\pd_{x_d}^nv(\cdot,t-s)}_{L^{\infty}(\R_+^{d})}}{1+|2^ks|^2}
\\
&\lec2^{-\frac{n}{2}k} \cdot M^{\cT}M^{\cT} \norm{\pd_{x_d}^nv(\cdot,t)}_{L^{\infty}(\R_+^{d})},
\end{align*}
which together with \eqref{eq:FS} yields \eqref{eq:trace1-2}.
\end{proof}

\begin{lem}\label{lem:trace-d0b}
Let $d\geq 2$. Suppose that  $g\in \mathcal{S}(\R^{d-1}\times \R)$ and let $v(x,t)$ be defined by \eqref{eq:def-v}. Then, for $1<p<\infty$ and $1\leq q\leq \infty$, as well as for $p=q=1$ and $p=q=\infty$, and for every integer $n\geq 0$, we have
\begin{align}\label{eq:trace2-1}
	\norm{\left(2^{(n-\frac{1}{q})m}\|\Delta_{m}^h\Delta_{<2m}^{\cT} g\|_{L^{q}(\R^{d-1})}\right)_{m\in\Z}}_{L^p(\R;\ell^q(\Z))}
		\lec \norm{\pd_{x_d}^nv}_{L^{p}\big(\R;\,L^{q}(\R_+^{d})\big)}.
\end{align}
Moreover, if $n>\frac{1}{q}$, then
	\begin{equation}\label{eq:trace2-2}
		\norm{\left(2^{(n-\frac{1}{q})m}\|\Delta_{m}^h\Delta_{\geq 2m}^{\cT} g\|_{L^{q}(\R^{d-1})}\right)_{m\in\Z}}_{L^p(\R;\ell^q(\Z))}
		\lec \norm{\pd_{x_d}^nv}_{L^{p}\big(\R;\,L^{q}(\R_+^{d})\big)}.
	\end{equation}
\end{lem}

The proof follows the same scheme as that of Lemma \ref{lem:trace-d0}. As before, we split the argument into three lemmas (Lemmas \ref{lem:trace-d1b}-\ref{lem:trace-d3b}), corresponding respectively to the three ranges of the indices $p$ and $q$.

\begin{lem}\label{lem:trace-d1b}
Under the assumptions of Lemma \ref{lem:trace-d0b} with $1<p,q<\infty$, estimates \eqref{eq:trace2-1} and \eqref{eq:trace2-2} hold.
\end{lem}

\begin{proof} 
We use a duality argument. 
For any sequence $\{\psi_m(x',t)\}_{m\in\mathbb Z}$ satisfying
\begin{equation*}
	\norm{\left(\norm{\psi_m(\cdot,t)}_{L^{q'}(\mathbb R^{d-1})}\right)_{m\in\Z}}_{L^{p'}(\mathbb R;\ell^{q'}(\Z))}<\infty,
\end{equation*}
we define
\begin{align*}
K_{41}&\eqdefa\sum_{m\in\mathbb{Z}}2^{(n-\frac{1}{q})m}\int_{\R^{d-1}\times\R} \Delta_{m}^{h}\Delta_{<2m}^{\mathcal{T}}g\cdot \psi_{m}\dx'\dt\\[3pt]
&=\sum_{m\in\mathbb{Z}}2^{(n-\frac{1}{q})m}\int_{\R^{d-1}\times\R} \Delta_{m}^{h}\Delta_{<2m}^{\mathcal{T}}g\cdot \Delta_{<2m+2}^{\mathcal{T}}\psi_{m}\dx'\dt,
\end{align*}
and
\begin{align*}
K_{42}&\eqdefa\sum_{m\in\mathbb{Z}}2^{(n-\frac{1}{q})m}\int_{\R^{d-1}\times\R} \Delta_{m}^{h}\Delta_{\geq2m}^{\mathcal{T}}g\cdot\psi_{m}\dx'\dt\\[3pt]
&=\sum_{m\in\mathbb{Z}}2^{(n-\frac{1}{q})m}\int_{\R^{d-1}\times\R} \sum_{j\geq 2m}\Delta_{m}^{h}\Delta_{j}^{\mathcal{T}}g\cdot\tilde{\Delta}_{j}^{\mathcal{T}}\psi_{m}\dx'\dt.
\end{align*}

We estimate $K_{41}$ first. By an argument similar to that used for \eqref{eq:hg1}, we get
\begin{align}\label{eq:hg1-1}
 \|\Delta_{m}^h\Delta_{<2m}^{\cT} g\|_{L^{q}(\R^{d-1})}&\lec\int_{0}^{\infty}\int_{\R} \frac{ x_d^\gamma 2^{(3+\gamma-n)m} e^{-cx_d 2^{m}}}{\left(1+|2^{2m} (t-s)|^2\right)^2}\norm{\pd_{x_d}^nv(\cdot,x_d,s)}_{L^{q}(\R^{d-1})}\ds\dx_d,
\end{align}
and hence
\begin{align}
	K_{41}
	&\lec\int_{\R}\int_{0}^{\infty}\sum_{m\in\mathbb{Z}}\Big( x_d^{\gamma}2^{(1+\gamma-\frac{1}{q})m} e^{-cx_d 2^{m}}\cdot\Big(\sup_{s\in \R} \frac{\norm{\Delta_{<2m+2}^{\mathcal{T}}\psi_{m}(\cdot,t-s)}_{L^{q'}(\R^{d-1})}}{1+|2^{2m}s|^2}\Big)\Big)\notag\\[3pt]
	&\quad\quad
	\times\norm{\pd_{x_d}^nv(\cdot,x_d,t)}_{L^{q}(\R^{d-1})}\dx_d\dt. \label{eq:K41}
\end{align}
Arguing as in \eqref{eq:K31}, and using the fact that
\begin{equation}\label{ML1}
\sup_{s\in \R} \frac{\norm{\Delta_{<2m+2}^{\mathcal{T}}\psi_{m}(\cdot,t-s)}_{L^{q'}(\R^{d-1})}}{1+|2^{2m}s|^2}
\lec M^{\cT}M^{\cT}\norm{\psi_m(\cdot,t)}_{L^{q'}(\R^{d-1})},
\end{equation}
we get by \eqref{eq:FS},
\begin{equation*}
	K_{41}\lec\norm{\left(\norm{\psi_{m}(\cdot,t)}_{L^{q'}(\R^{d-1})}\right)_{m\in\Z}}_{L^{p'}(\R;\ell^{q'}(\Z))}\cdot\norm{\pd_{x_d}^nv}_{L^{p}(\R;L^q\left(\R_+^d\right))}.
\end{equation*}
By a duality argument, it follows that
for $1<p<\infty$ and $1< q<\infty$, estimate \eqref{eq:trace2-1} holds.

We now estimate $K_{42}$. For $j\geq 2m$, arguing as in \eqref{eq:hg2}, we get
\begin{equation}\label{eq:hg2-2}
	\|\Delta_{m}^{h}\Delta_{j}^{\cT} g\|_{L^{q}(\R^{d-1})}\lec\int_{0}^{\infty} \!\!\int_{\R}\frac{ x_d^{\gamma}2^{(3+\gamma-n)\frac{j}{2}}e^{-cx_d 2^{\frac{j}{2}}}}{\left(1+|2^j (t-s)|^2\right)^2}\norm{\pd_{x_d}^nv(\cdot,x_d,s)}_{L^{q}(\R^{d-1})}\ds\dx_d.
\end{equation}
Hence, we obtain 
\begin{align*}
	K_{42}&\lec\sum_{m\in\mathbb{Z}}\sum_{j\geq 2m}\int_{\R}\Big(\int_{0}^{\infty}\int_{\R}
	\frac{2^{(n-\frac{1}{q})m+(3+\gamma-n)\frac{j}{2}} x_d^{\gamma}e^{-cx_d 2^{\frac{j}{2}}}}{\left(1+|2^j (t-s)|^2\right)^2}\cdot\norm{\pd_{x_d}^nv(\cdot,x_d,s)}_{L^{q}(\R^{d-1})}\ds\dx_d\Big) \notag\\[3pt]
	&\quad\quad \times\norm{\tilde{\Delta}_{j}^{\mathcal{T}}\psi_{m}(\cdot,t)}_{L^{q'}(\R^{d-1})}\dt \notag\\[3pt]
	&\lec\int_{\R}\int_{0}^{\infty}\sum_{m \in\mathbb{Z}}\sum_{j\geq 2m}\Big( 2^{(n-\frac{1}{q})m+(1+\gamma-n)\frac{j}{2}}x_d^{\gamma} e^{-cx_d 2^{\frac{j}{2}}}\cdot\Big(\sup_{s\in \R} \frac{\norm{\tilde{\Delta}_{j}^{\mathcal{T}}\psi_{m}(\cdot,t-s)}_{L^{q'}(\R^{d-1})}}{1+|2^js|^2}\Big)\Big)\notag\\[3pt]
	&\quad\quad
	\times\norm{\pd_{x_d}^nv(\cdot,x_d,t)}_{L^{q}(\R^{d-1})}\dx_d\dt.
\end{align*}
Bounding Peetre's maximal function by the Hardy-Littlewood maximal function, and then applying \eqref{eq:gamma-q-n} with $1+\gamma-n\leq0$ (as $n>1/q$), we get
\begin{align}
	K_{42}&\lec\int_{\R}\int_{0}^{\infty}\sum_{m \in\mathbb{Z}}\Big( 2^{(1+\gamma-\frac{1}{q}-\ep)m}x_d^{\gamma-\ep} e^{-\frac{c}{2}x_d 2^{m}}\cdot M^{\cT}M^{\cT}\norm{\psi_m(\cdot,t)}_{L^{q'}(\R^{d-1})}\Big)\notag\\[3pt]
	&\quad\quad
	\times\norm{\pd_{x_d}^nv(\cdot,x_d,t)}_{L^{q}(\R^{d-1})}\dx_d\dt.
	\label{eq:K42}
\end{align}
Then applying H\"older's inequality and then \eqref{eq:int-xd1} with $\ep=\frac{1}{2q'^2}$, we obtain the same bound as $K_{41}$ for $K_{42}$. The estimate \eqref{eq:trace2-2} follows in the same way.
\end{proof}

\begin{lem}\label{lem:trace-d2b}
Under the assumptions of Lemma \ref{lem:trace-d0b} with $1\leq p<\infty$ and $q=1$, estimates \eqref{eq:trace2-1} and \eqref{eq:trace2-2} hold.
\end{lem}
\begin{proof}
The proof is similar to that of Lemma \ref{lem:trace-d2}. By \eqref{eq:K41} with $q=1$ and \eqref{ML1}, we get
\begin{equation*}
	K_{41}\lec\int_{\R}\sup_{m}M^{\cT}M^{\cT}\norm{\psi_m(\cdot,t)}_{L^{\infty}(\R^{d-1})}\cdot\norm{\pd_{x_d}^nv(\cdot,t)}_{L^{1}(\R_+^{d})}\dt.  \label{eq:K31-q=1}
\end{equation*}
Hence by \eqref{eq:FS} and a duality argument, we obtain \eqref{eq:trace2-1}.

By \eqref{eq:K42} with  $\ep=\frac{1}{2}$, we obtain the same bound as $K_{41}$ for $K_{42}$. The estimate \eqref{eq:trace2-2} follows in the same way.
\end{proof}
\begin{lem}\label{lem:trace-d3b}
	Under the assumptions of Lemma \ref{lem:trace-d0b} with $1< p\leq\infty$ and $q=\infty$, estimates \eqref{eq:trace2-1} and \eqref{eq:trace2-2} hold.
\end{lem}
\begin{proof}
The proof is similar to that of Lemma \ref{lem:trace-d3}. Analogously to \eqref{eq:hg1-1} and \eqref{4.24}, it is easy to verify that
\begin{align*}
	&\quad \ \|\Delta_{m}^h\Delta_{<2m}^{\cT} g\|_{L^{\infty}(\R^{d-1})}\\[3pt]
	&\lec\int_{0}^{\infty}\int_{\R} \frac{ 2^{(3-n)m} e^{-cx_d 2^{m}}}{\left(1+|2^{2m} (t-s)|^2\right)^2}\norm{\Delta_{<2m}^{\mathcal{T}}\pd_{x_d}^nv(\cdot,x_d,s)}_{L^{\infty}(\R^{d-1})}\ds\dx_d\notag \\[3pt]
	&\lec\int_{0}^{\infty} 2^{(1-n)m} e^{-cx_d 2^{m}}\cdot\sup_{s\in \R} \frac{\norm{\Delta_{<2m}^{\mathcal{T}}\pd_{x_d}^nv(\cdot,x_d,t-s)}_{L^{\infty}(\R^{d-1})}}{1+|2^{2m}s|^2}\dx_d\notag \\[3pt]
	&\lec2^{-nm}\sup_{s\in \R} \frac{\norm{\Delta_{<2m}^{\mathcal{T}}\pd_{x_d}^nv(\cdot,t-s)}_{L^{\infty}(\R_+^{d})}}{1+|2^{2m}s|^2}\notag\\[3pt]
	&\lec2^{-nm} \cdot M^{\cT}M^{\cT} \norm{\pd_{x_d}^nv(\cdot,t)}_{L^{\infty}(\R_+^{d})}, 
\end{align*}
which together with \eqref{Maximal1} yields \eqref{eq:trace2-1}.

By \eqref{eq:hg2-2}, \eqref{eq:lem-xd5}, \eqref{eq:peetre} with $r=1$ and \eqref{eq:maximal-uniform}, we achieve
\begin{align*}
	 &\quad\ \sum_{j\geq2m}\|\Delta_{m}^h\Delta_{j}^{\cT} g\|_{L^{\infty}(\R^{d-1})}\\[3pt]
	&\lec\sum_{j\geq2m}\int_{0}^{\infty}\int_{\R} \frac{  2^{(3-n)\frac{j}{2}}e^{-cx_d 2^{\frac{j}{2}}}}{\left(1+|2^j (t-s)|^2\right)^2}\norm{\Delta_{j}^{\mathcal{T}}\pd_{x_d}^nv(\cdot,x_d,s)}_{L^{\infty}(\R^{d-1})}\ds\dx_d \\[3pt]
	&\lec\sum_{j\geq2m}\int_{0}^{\infty} 2^{(1-n)\frac{j}{2}} e^{-cx_d 2^{\frac{j}{2}}}\cdot\sup_{s\in \R} \frac{\norm{\Delta_{j}^{\mathcal{T}}\pd_{x_d}^nv(\cdot,x_d,t-s)}_{L^{\infty}(\R^{d-1})}}{1+|2^js|^2}\dx_d \notag \\[3pt]
	&\lec\sum_{j\geq2m} 2^{-n\frac{j}{2}} \cdot\sup_{s\in \R} \frac{\norm{\Delta_{j}^{\mathcal{T}}\pd_{x_d}^nv(\cdot,t-s)}_{L^{\infty}(\R_+^{d})}}{1+|2^js|^2} \notag \\[3pt]
	&\lec2^{-nm} \cdot M^{\cT}M^{\cT} \norm{\pd_{x_d}^nv(\cdot,t)}_{L^{\infty}(\R_+^{d})}.
\end{align*}
which together with \eqref{Maximal1} yields \eqref{eq:trace2-2}.
\end{proof}

\begin{proof}[Proof of Theorem \ref{thm:dd-trace}]

The first term on the left-hand side of \eqref{thm:dd-trace-eq1} is bounded by its right-hand side by summing
the two estimates in Lemma \ref{lem:trace-d0}.
The second term on the left-hand side of \eqref{thm:dd-trace-eq1} is treated similarly using Lemma \ref{lem:trace-d0b}.
\end{proof}

\section*{Acknowledgments}
We warmly thank Professor Senjo Shimizu and Ting Zhang for fruitful discussions. TT also thanks colleagues at Kyoto University for their hospitality during his visit in June 2025.

Chen was supported in part by National Natural Science Foundation of China under grant [12101556]. 
The research of both Liang and Tsai was partially supported by Natural Sciences and Engineering Research Council of Canada (NSERC) under grant RGPIN-2023-04534.

\bibliography{MaxReg}
\bibliographystyle{abbrv}

\end{document}